%% file: triangle_revised.tex
\documentclass[11pt, reqno]{article}
\input{header}

\input{commands}

\usepackage[margin=0.95in]{geometry}
\usepackage{extarrows}
\usepackage{titling}
\usepackage{commath}
\usepackage{mathtools}
\usepackage{titlesec}
\newcommand{\eps}{\varepsilon}
\usepackage{bbm}
\usepackage{setspace}
\usepackage{enumitem}
\usepackage{tikz,pgfplots}

\usepackage[textsize=tiny]{todonotes}

\usepackage{cite}

\newcommand{\Aut}{\operatorname{Aut}}

\newcommand{\lrDini}[1]{\left(\frac{d}{d #1}\right)_{\hspace{-0.2em}+}\!}

\newcommand{\myfrac}[3][0pt]{\genfrac{}{}{}{}{\raisebox{#1}{$#2$}}{\raisebox{-#1}{$#3$}}}

\def\P{\mathbb{P}}

\newcommand{\Chi}{\mathcal{X}}

\DeclareMathSymbol{\leqslant}{\mathalpha}{AMSa}{"36} 
\DeclareMathSymbol{\geqslant}{\mathalpha}{AMSa}{"3E} 
\DeclareMathSymbol{\eset}{\mathalpha}{AMSb}{"3F}     

\renewcommand{\epsilon}{\varepsilon}

\makeatletter
\pgfdeclareshape{slash underlined}
{
  \inheritsavedanchors[from=rectangle] 
  \inheritanchorborder[from=rectangle]
  \inheritanchor[from=rectangle]{north}
  \inheritanchor[from=rectangle]{north west}
  \inheritanchor[from=rectangle]{north east}
  \inheritanchor[from=rectangle]{center}
  \inheritanchor[from=rectangle]{west}
  \inheritanchor[from=rectangle]{east}
  \inheritanchor[from=rectangle]{mid}
  \inheritanchor[from=rectangle]{mid west}
  \inheritanchor[from=rectangle]{mid east}
  \inheritanchor[from=rectangle]{base}
  \inheritanchor[from=rectangle]{base west}
  \inheritanchor[from=rectangle]{base east}
  \inheritanchor[from=rectangle]{south}
  \inheritanchor[from=rectangle]{south west}
  \inheritanchor[from=rectangle]{south east}
  \inheritanchorborder[from=rectangle]
  \foregroundpath{
    \southwest \pgf@xa=\pgf@x \pgf@ya=\pgf@y
    \northeast \pgf@xb=\pgf@x \pgf@yb=\pgf@y
    \pgf@xc=\pgf@xa
    \advance\pgf@xc by .5\pgf@xb
    \pgf@yc=\pgf@ya
    \advance\pgf@xc by -1.3pt
    \advance\pgf@yc by -1.8pt
    \pgfpathmoveto{\pgfqpoint{\pgf@xc}{\pgf@yc}}
    \advance\pgf@xc by  2.6pt
    \advance\pgf@yc by  3.6pt
    \pgfpathlineto{\pgfqpoint{\pgf@xc}{\pgf@yc}}
    \pgfpathmoveto{\pgfqpoint{\pgf@xa}{\pgf@ya}}
    \pgfpathlineto{\pgfqpoint{\pgf@xb}{\pgf@ya}}
 }
}
\tikzset{nomorepostaction/.code=\let\tikz@postactions\pgfutil@empty}

\makeatother

\makeatletter
\DeclareFontFamily{OMX}{MnSymbolE}{}
\DeclareSymbolFont{MnLargeSymbols}{OMX}{MnSymbolE}{m}{n}
\SetSymbolFont{MnLargeSymbols}{bold}{OMX}{MnSymbolE}{b}{n}
\DeclareFontShape{OMX}{MnSymbolE}{m}{n}{
    <-6>  MnSymbolE5
   <6-7>  MnSymbolE6
   <7-8>  MnSymbolE7
   <8-9>  MnSymbolE8
   <9-10> MnSymbolE9
  <10-12> MnSymbolE10
  <12->   MnSymbolE12
}{}
\DeclareFontShape{OMX}{MnSymbolE}{b}{n}{
    <-6>  MnSymbolE-Bold5
   <6-7>  MnSymbolE-Bold6
   <7-8>  MnSymbolE-Bold7
   <8-9>  MnSymbolE-Bold8
   <9-10> MnSymbolE-Bold9
  <10-12> MnSymbolE-Bold10
  <12->   MnSymbolE-Bold12
}{}

\let\llangle\@undefined
\let\rrangle\@undefined
\DeclareMathDelimiter{\llangle}{\mathopen}%
                     {MnLargeSymbols}{'164}{MnLargeSymbols}{'164}
\DeclareMathDelimiter{\rrangle}{\mathclose}%
                     {MnLargeSymbols}{'171}{MnLargeSymbols}{'171}
\makeatother

\pretitle{\begin{flushleft}\Large}
\posttitle{\par\end{flushleft}\vskip 0.5em
}
\preauthor{\begin{flushleft}}
\postauthor{
\\
\vspace{0.5em}
\footnotesize{The Division of Physics, Mathematics and Astronomy, California Institute of Technology.\\ Email: 
\href{mailto:t.hutchcroft@caltech.edu}{t.hutchcroft@caltech.edu}}
\end{flushleft}}
\predate{\begin{flushleft}}
\postdate{\par\end{flushleft}}
\title{\bf On the derivation of mean-field percolation critical exponents from the triangle condition
}

\renewenvironment{abstract}
 {\par\noindent\textbf{\abstractname.}\ \ignorespaces}
 {\par\medskip}

\titleformat{\section}
  {\normalfont\large \bf}{\thesection}{0.4em}{}

\newcommand{\Tr}{\operatorname{Tr}}

\author{{\bf Tom Hutchcroft}}
\begin{document}

\date{\small{\today}}

\maketitle

\setstretch{1.1}

\begin{abstract}
We give a new derivation of mean-field percolation critical behaviour from the triangle condition that is 
 quantitatively much better than previous proofs when the triangle diagram $\nabla_{p_c}$ is large. In contrast to earlier methods, our approach continues to yield bounds of reasonable order when the triangle diagram $\nabla_p$ is unbounded but diverges slowly as $p \uparrow p_c$, as is expected to occur in percolation on $\Z^d$ at the upper-critical dimension $d=6$. Indeed, we show in particular that if the triangle diagram diverges polylogarithmically as $p \uparrow p_c$ then mean-field critical behaviour holds to within a polylogarithmic  factor. We apply the methods we develop to deduce that for long-range percolation on the hierarchical lattice, mean-field critical behaviour holds to within polylogarithmic factors at the upper-critical dimension. 

 As part of the proof, we introduce a new method for comparing diagrammatic sums on general transitive graphs that may be of independent interest.
\end{abstract}



\tableofcontents



\section{Introduction}\label{sec:intro}

\subsection{A new differential inequality for the susceptibility.}

\emph{The triangle condition} is a sufficient condition for mean-field critical behaviour in Bernoulli percolation models that was introduced by Aizenman and Newman \cite{MR762034} and proven to hold on high-dimensional Euclidean lattices in the milestone work of Hara and Slade \cite{MR1043524}. Let $G=(V,E)$ be a connected, locally finite, transitive graph, let $o$ be a fixed vertex of $G$, and for each $0\leq p \leq 1$ and $x,y\in V$ let $T_p(x,y)=\P_p(x\leftrightarrow y)$ be the probability that $x$ and $y$ are connected in Bernoulli-$p$ bond percolation. We think of $T_p$ as a matrix indexed by $V\times V$ and refer to it as the \textbf{two-point matrix}. For each $0\leq p \leq 1$ we define the \textbf{triangle diagram} $\nabla_p$ to be
\[
\nabla_p := \sum_{x,y \in V} T_p(o,x)T_p(x,y)T_p(y,o)=T_p^3(o,o).
\]
We say that $G$ satisfies the \textbf{triangle condition} if $\nabla_{p}$ is finite when $p$ is equal to the critical probability $p_c=p_c(G)$. Aizenman and Newman \cite{MR762034} and Barsky and Aizenman \cite{MR1127713} proved that if $G$ is, say, a Cayley graph satisfying the triangle condition then $G$ has mean-field critical behaviour for percolation in the sense that if $K$ denotes the cluster of the origin then
\begin{align}
\E_p |K| &\asymp |p-p_c|^{-1} &&\text{ as $p\uparrow p_c$,}
\label{eq:mean_field_susceptibility}\\
\P_{p_c}(|K|\geq n) &\asymp n^{-1/2} &&\text{ as $n \uparrow\infty$, and}
\label{eq:mean_field_volume}
\\
\P_p(|K|=\infty) &\asymp |p-p_c| &&\text{ as $p\downarrow p_c$,}
\label{eq:mean_field_density}
\end{align}
where we write $\asymp$ to denote an equality holding to within multiplication by positive constants within the vicinity of the relevant limit point. Note that the \emph{lower bounds} of \eqref{eq:mean_field_susceptibility} and \eqref{eq:mean_field_density} are known to hold on every transitive graph \cite{aizenman1987sharpness,duminil2015new}. The triangle condition has now been proven to hold in a variety of high-dimensional settings \cite{fitzner2015nearest,hutchcroft2021critical,1804.10191,MR1833805,Hutchcroftnonunimodularperc,MR2165583,MR3306002,MR2430773}, and 
further works studying critical behaviour in high-dimensional percolation under the triangle condition include \cite{MR2551766,MR2748397,MR923855,dewan2021upper,chatterjee2018restricted,MR2155704}; see \cite{heydenreich2015progress} for an overview of this extensive literature and \cite{grimmett2010percolation} for further background on percolation.


The proofs of \cref{eq:mean_field_susceptibility,eq:mean_field_volume,eq:mean_field_density} rely crucially on \emph{differential inequalities}. Let us consider as a paradigmatic example the rate of growth of the \emph{susceptibility} $\Chi_p:=\E_p |K|$ as $p\uparrow p_c$, whose behaviour under the triangle condition was determined by Aizenman and Newman \cite{MR762034}. It is a fundamental fact known as the \emph{sharpness of the phase transition} and due originally to Menshikov \cite{MR852458} and Aizenman and Barsky \cite{aizenman1987sharpness} that $\Chi_p < \infty$ if and only if $p<p_c$; see also \cite{duminil2015new,MR3898174,1901.10363} for alternative proofs. Moreover, the susceptibility $\Chi_p$ is a smooth function of $p$ on $[0,p_c)$ \cite[Section 6.4]{grimmett2010percolation} and satisfies the differential inequality
\begin{equation}
\label{eq:Chi_prime_upper}
\frac{d\Chi_p}{dp} \leq 2\deg(o) \Chi_p^2
\end{equation}
for every $0<p<p_c$ \cite[Equation 4.2.4]{heydenreich2015progress} (in fact the factor $2$ is not needed). This inequality can be integrated to obtain that the \emph{mean-field lower bound}
\begin{equation}
\label{eq:Chi_lower}
\Chi_p \geq \frac{1}{2\deg(o)} |p-p_c|^{-1}
\end{equation}
holds for every $0\leq p <p_c$. In order to establish mean-field critical behaviour under the triangle condition, the key step is to establish a \emph{lower bound} on the derivative of $\Chi_p$ of the same order as the upper bound of \eqref{eq:Chi_prime_upper}, which can then be integrated to obtain an upper bound on $\Chi_p$ of the same order as the lower bound of \eqref{eq:Chi_lower}. To do this, one  considers the so-called \emph{open triangle} $\nabla_{p_c}(r) = \sup_{x:d(o,x) \geq r} T_{p_c}^3(o,x)$, which tends to zero as $r \to \infty$ when the triangle condition holds \cite{MR2779397}, and argues that for each $r \geq 1$ there exists $\eps(r)>0$ such that
\begin{equation}
\label{eq:bad_old_diff_ineq}
\frac{d\Chi_p}{dp} \geq \eps(r) (1-\nabla_{p_c}(r)) \Chi_p^2
\end{equation}
for every $p_c/2 \leq p<p_c$. See also the proof of \cite[Theorem 10.68]{grimmett2010percolation} for a detailed treatment of the case $G=\Z^d$. Taking $r$ to be sufficiently large that $\nabla_{p_c}(r)<1$ gives a differential inequality of the required form. Unfortunately, the dependence of $\eps(r)$ on $r$ given by these arguments tends to be very bad (e.g.\ exponentially small in $r$), so that the implicit constants appearing in \cref{eq:mean_field_susceptibility,eq:mean_field_volume,eq:mean_field_density} tend to be extremely large when the triangle sum $\nabla_{p_c}$ is finite but large. The same shortcomings make these methods poorly suited to situations where, say, $\nabla_{p_c}$ is infinite but $\nabla_{p}$ diverges slowly as $p \uparrow p_c$, as is expected to happen on $\Z^d$ at the upper-critical dimension $d=6$.


In this paper we give a new derivation of mean-field critical behaviour from the triangle condition that is arguably simpler than the standard method (in the case of $\Z^d$) and that gives much better quantitative control when the triangle sum is large. We then apply this new method to study long-range percolation on the \emph{hierarchical lattice} at the upper-critical dimension as discussed in detail in \cref{subsec:hierarchical_intro}. 
Since we are interested in applying our results to both long-range and nearest-neighbour models, we work with percolation on \emph{weighted graphs}, which we now briefly define. A \textbf{weighted graph} $G=(V,E,J)$ is a countable graph $(V,E)$ together with an assignment of non-negative \textbf{weights} 
 $\{J_e : e \in E\}$ such that $\sum_{e\in E^\rightarrow_v} J_e < \infty$ for each $v\in V$, where $E^\rightarrow_v$ denotes the set of oriented edges of $G$ emanating from the vertex $v$. (Although all our graphs are unoriented, it is useful to think of each unoriented edge as corresponding to a pair of oriented edges pointing in opposite directions.) Locally finite graphs can be considered as weighted graphs by setting $J_e \equiv 1$. A graph automorphism of $(V,E)$ is a weighted graph automorphism of $(V,E,J)$ if it preserves the weights, and a weighted graph $G$ is said to be \textbf{transitive} if any vertex can be mapped to any other vertex by an automorphism. Given a weighted graph $G=(V,E,J)$ and $\beta\geq 0$, we define Bernoulli-$\beta$ bond percolation on $G$ to be the random subgraph of $G$ in which each edge $e$ is chosen to be either retained or deleted independently at random with retention probability $1-e^{-\beta J_e}$, and write $\P_\beta=\P_{G,\beta}$ for the law of this random subgraph. 


 The main contribution of the paper is the following simple explicit strengthening of the differential inequality \eqref{eq:bad_old_diff_ineq}. Throughout the paper we will assume without loss of generality that $\sum_{e\in E^\rightarrow_o} J_e=1$; multiplying all edge weights by a constant so that this holds is equivalent to multiplying the parameter $\beta$ by a constant. As above, when $G$ is transitive we fix a vertex $o$ of $G$, write $K$ for the cluster of $o$, and write $\Chi_\beta=\E_\beta |K|$ for the susceptibility, which is a finite-valued, smooth function of $\beta$ on $[0,\beta_c)$ by sharpness of the phase transition \cite{aizenman1987sharpness,duminil2015new,1901.10363}.

\begin{theorem}
\label{thm:main_chi}
Let $G$ be a connected, unimodular, transitive weighted graph that is normalised so that $\sum_{e\in E^\rightarrow_o} J_e=1$. Then
\begin{equation}
\label{eq:mainchi}
\frac{d \Chi_\beta}{d\beta} \geq  \frac{\Chi_\beta(\Chi_\beta-\nabla_\beta)}{3 \beta^2 \nabla_\beta^2}
\end{equation}
for every $0\leq \beta < \beta_c$.
\end{theorem}

Here, we recall that a connected, transitive weighted graph $G=(V,E,J)$ is said to be \textbf{unimodular} if it satisfies the \textbf{mass-transport principle}, meaning that if $F:V\times V\to [0,\infty]$ is diagonally invariant in the sense that $F(\gamma x,\gamma y)=F(x,y)$ for every $x,y \in V$ and $\gamma \in \Aut(G)$ then
\begin{equation}
\label{eq:MTP}
\sum_{x\in V} F(o,x)=\sum_{x\in V} F(x,o).
\end{equation}
Most transitive weighted graphs arising in applications are unimodular, including all amenable transitive weighted graphs and all Cayley graphs of countable groups \cite{MR1082868}. Indeed, if $G$ is a Cayley graph of a countable group $\Gamma$ then left multiplication by an element of $\Gamma$ defines an automorphism of $G$ and we trivially verify \eqref{eq:MTP} by writing
\[
\sum_{\gamma\in \Gamma} F(\mathrm{id},\gamma) = \sum_{\gamma \in \Gamma} F(\gamma^{-1}\mathrm{id},\gamma^{-1} \gamma) =\sum_{\gamma \in \Gamma} F(\gamma^{-1},\mathrm{id}) = \sum_{\gamma \in \Gamma} F(\gamma,\mathrm{id}).
\]
(Note that all \emph{nonunimodular} transitive graphs are proven to have mean-field critical behaviour for  percolation in \cite{Hutchcroftnonunimodularperc} via different methods that are exclusive to the nonunimodular case.)

\medskip

Integrating the differential inequality \eqref{eq:mainchi} yields the following corollary, which recovers the estimate \eqref{eq:mean_field_susceptibility} when the triangle condition holds. 

\begin{corollary}
\label{cor:integrated}
Let $G$ be a connected, unimodular, transitive weighted graph that is normalised so that $\sum_{e\in E^\rightarrow_o} J_e=1$. Then
\[
\Chi_\beta \leq \left(1+\frac{\beta_c-\beta}{3\beta \beta_c}  \right) \left[ \int_\beta^{\beta_c} \frac{1}{3 \lambda^2 \nabla_\lambda^2} \dif \lambda \right]^{-1}
\]
for every $0\leq \beta < \beta_c$.
\end{corollary}

Note that the $(1+(\beta_c-\beta)/(3\beta \beta_c))$ prefactor appearing here converges to $1$ as $\beta \uparrow \beta_c$.


\begin{proof}[Proof of \cref{cor:integrated} given \cref{thm:main_chi}]
Let $0\leq \beta < \beta_c$. Since $\Chi_{{\beta_c}}=\infty$ we have that
\begin{align*}
\frac{1}{\Chi_\beta}=\frac{1}{\Chi_\beta}-\frac{1}{\Chi_{\beta_c}} = -\int_\beta^{\beta_c} \frac{d}{d \lambda}  \frac{1}{\Chi_\lambda}\dif \lambda = \int_\beta^{\beta_c} \frac{1}{\Chi_\lambda^2} \frac{d \Chi_\lambda}{d \lambda}\dif \lambda.
\end{align*}
It follows from \cref{thm:main_chi} that
\begin{align*}
\frac{1}{\Chi_\beta}\geq  \int_\beta^{\beta_c} \frac{1}{3\lambda^2 \nabla_\lambda^2} \dif \lambda - \int_\beta^{\beta_c} \frac{1}{3\lambda^2 \nabla_\lambda \Chi_\lambda} \dif \lambda \geq \int_\beta^{\beta_c} \frac{1}{3\lambda^2 \nabla_\lambda^2} \dif \lambda - \frac{1}{\Chi_\beta} \int_\beta^{\beta_c} \frac{1}{3\lambda^2} \dif \lambda,
\end{align*}
where we used that $\Chi_\lambda$ is increasing in $\lambda$ and that $\nabla_\lambda \geq 1$ for every $\lambda \geq 0$ in the second inequality. The claim follows by rearranging this inequality.
\end{proof}


\subsection{Other exponents via the relative entropy method}
Classically, Barsky and Aizenman \cite{MR1127713} proved that \eqref{eq:mean_field_volume} and \eqref{eq:mean_field_density} hold under the triangle condition by proving a further partial differential inequality concerning the \emph{magnetization} $M_{\beta,h}=\E_\beta [1-e^{-h|K|}] \asymp h\sum_{n=1}^{\lceil 1/h \rceil} \P_\beta(|K|\geq n)$. Specifically, they proved that there exists a constant $C$ such that for each $r\geq 1$ there exists $\eps(r)>0$ such that
\begin{equation}
\label{eq:BarskyAizenmanPDI}
M_{\beta,h}-h\frac{\partial M_{\beta,h}}{\partial h} \geq \eps(r)\Bigl(1-C\nabla_\beta(r)\Bigr) 
M_{\beta,h}^2 \frac{\partial M_{\beta,h}}{\partial h}
- \frac{1}{\eps(r)} h M_{\beta,h} \frac{\partial M_{\beta,h}}{\partial h}
\end{equation}
for every $\beta_c /2 \leq \beta <\beta_c$ and $h\geq 0$; the upper bound of \eqref{eq:mean_field_volume} can be deduced from this partial differential inequality via an elementary (if non-obvious) calculation. This method suffers from the same quantitative shortcomings affecting Aizenman and Newman's analysis of the susceptibility and does not provide estimates of reasonable order when the triangle sum is unbounded but diverges slowly as $\beta \uparrow \beta_c$.

\medskip

Rather than following this strategy, we instead observe that the upper bounds of \eqref{eq:mean_field_volume} and \eqref{eq:mean_field_density} can each be deduced directly from the upper bound of \eqref{eq:mean_field_susceptibility} using the \emph{relative entropy method} pioneered in the recent work of Dewan and Muirhead~\cite{dewan2021upper}, which yields good quantitative control of all the relevant constants. This allows us to completely avoid proving any partial differential inequality of the form \eqref{eq:BarskyAizenmanPDI} and obtain versions of \eqref{eq:mean_field_volume} and \eqref{eq:mean_field_density} with good quantitative dependence on the value of the triangle diagram $\nabla_{p_c}$. More specifically, we adapt the methods of Dewan and Muirhead to prove the following general inequality.

\begin{thm}
\label{thm:DM_Chi}
Let $G=(V,E,J)$ be a connected, transitive weighted graph that is normalised so that $\sum_{e\in E^\rightarrow_o} J_e=1$. Then
\[
\P_{\beta_2}(|K|\geq n) 
 \leq  \left(\frac{2}{n} +\frac{4}{\beta_1}|\beta_2-\beta_1|^2\right) \Chi_{\beta_1}
\]
for every $\beta_2 \geq \beta_1 >0$ and $n\geq 0$.
\end{thm}

Taking $\beta_1=\beta_c-n^{-1/2}$ and $\beta_2=\beta_c$ in this inequality shows that the upper bound of \eqref{eq:mean_field_susceptibility} implies the upper bound of \eqref{eq:mean_field_volume}, while taking $\beta_1=\beta_c-\eps$, $\beta_2=\beta_c+\eps$ and $n\uparrow \infty$ shows that the upper bound of \eqref{eq:mean_field_susceptibility} implies the upper bound of \eqref{eq:mean_field_density}. The complementary fact that the upper bound of \eqref{eq:mean_field_volume} implies the upper bound of \eqref{eq:mean_field_susceptibility} was established in \cite[Theorem 1.1]{1901.10363}. 

\begin{remark}
The proof of \cref{thm:DM_Chi} will in fact establish the slightly stronger inequality
\begin{equation}
\label{eq:DM_Chi_strong}
\P_{\beta_2}(|K|\geq n) \leq 2\P_{\beta_1}(|K|\geq n) + \frac{4}{\beta_1}|\beta_2-\beta_1|^2 \sum_{k=1}^n \P_{\beta_1}(|K|\geq k),
\end{equation}
from which \cref{thm:DM_Chi} follows by Markov's inequality to the term $\P_{\beta_1}(|K|\geq n)$ and noting that $\sum_{k=1}^n \P_{\beta_1}(|K|\geq k) \leq \sum_{k=1}^\infty \P_{\beta_1}(|K|\geq k)=\Chi_{\beta_1}$.
Taking $\beta_1=\beta_c$, $\beta_2=\beta_c+\eps$, and $n=\lceil \eps^{-2} \rceil$ in this inequality we can also deduce directly that the upper bound of \eqref{eq:mean_field_volume} implies the upper bound of \eqref{eq:mean_field_density}. The same implication may also be proven using the \emph{extrapolation} technique of Aizenman and Fernandez \cite{MR857063} as discussed in \cite{MR1127713}.
\end{remark}

\begin{remark}
The bounds of \cref{thm:DM_Chi} are similar to, but quantitatively better than, those appearing in the work of Newman \cite{MR912497,MR869320}, which relies on differential inequalities obtained via large-deviations analysis of the \emph{fluctuation}
 $(1-p) \#\{$open edges in $K\}-p\#\{$closed edges touching $K\}$. One can sharpen Newman's analysis by using maximal inequalities instead of large-deviations estimates but the bounds one obtains this way are still not as strong as those of \cref{thm:DM_Chi}. 
As explained to us by Stephen Muirhead, it appears to be a rather general phenomenon that relative entropy methods give  `non-differential improvements' to estimates based on the analysis of the fluctuation.
\end{remark}






The following corollary is illustrative of what can be done with the new methods we introduce. The case $\alpha=0$ of the corollary recovers the classical fact the triangle condition implies mean-field critical behaviour as proven by Aizenman and Newman \cite{MR762034} and Barsky and Aizenman \cite{MR1127713}. 
We write $\preceq$ and $\succeq$ for inequalities holding to within constant positive multiplicative factors in a neighbourhood of the relevant limit point.

\begin{corollary}
\label{cor:logs}
Let $G$ be a connected, unimodular, transitive weighted graph and suppose that there exists $\alpha \geq 0$ and $A<\infty$ such that $\nabla_{\beta_c-\eps} \preceq (\log (1/\eps))^\alpha$ as $\eps \downarrow 0$. Then
\begingroup
\addtolength{\jot}{0.5em}
\begin{align}
\Chi_{\beta_c-\eps} &\preceq \Bigl(\log \frac{1}{\eps}\Bigr)^{2\alpha}\cdot\frac{1}{\eps} &\text{ as }\eps &\downarrow 0,
\label{logexponent:susceptibility}
\\
\P_{\beta_c}\left( |K| \geq n\right) &\preceq \hspace{0.13em} (\hspace{0.2em}\log n\hspace{0.15em})^{2\alpha}  \cdot \frac{1}{n^{1/2}} &\text{ as }n&\uparrow \infty,  \text{ and}
\label{logexponent:volume}\\
\P_{\beta_c+\eps}\left(|K|=\infty\right) &\preceq  \Bigl(\log \frac{1}{\eps}\Bigr)^{2\alpha}\cdot \eps&\text{ as }\eps &\downarrow 0.
\label{logexponent:theta}
\end{align}
\endgroup
In particular, mean-field critical behaviour holds to within polylogarithmic factors.
\end{corollary}

See \cite[Proposition 3.1]{aizenman1983renormalized} and \cite[Theorem 1.5.4]{MR2986656} for related results for the four-dimensional Ising model and self-avoiding walk.
The hypothesis $\nabla_{\beta_c-\eps} \preceq (\log (1/\eps))^\alpha$ is expected to hold with $\alpha>0$ for nearest-neighbour percolation on $\Z^d$ at the upper-critical dimension $d=6$, although proving this appears to be completely beyond the scope of existing methods. 
Indeed, it is believed that for percolation models at the upper-critical dimension there should exist non-zero exponents $\theta_\gamma$, $\theta_\delta$, and $\theta_\beta$ such that 
\begingroup
\begin{align}
\Chi_{\beta_c-\eps} &\asymp \Bigl(\log \frac{1}{\eps}\Bigr)^{\theta_\gamma}\cdot\frac{1}{\eps} &\text{ as }\eps &\downarrow 0,
\\
\P_{\beta_c}\left( |K| \geq n\right) &\asymp \hspace{0.13em} (\hspace{0.2em}\log n\hspace{0.15em})^{\hspace{0.13em}\theta_\delta}  \cdot \frac{1}{n^{1/2}} &\text{ as }n&\uparrow \infty,  \text{ and}
\\
\P_{\beta_c+\eps}\left(|K|=\infty\right) &\asymp  \Bigl(\log \frac{1}{\eps}\Bigr)^{\theta_\beta}\cdot \eps&\text{ as }\eps &\downarrow 0.
\end{align}
\endgroup
For nearest neighbour percolation on $\Z^d$ at the upper-critical dimension $d=6$, Essam, Gaunt, and Guttmann \cite{essam1978percolation} predicted that these asymptotics hold with $\theta_\gamma=\theta_\delta=\theta_\beta=2/7$. (Their notation is different to ours: their $\theta_\delta$ is $2\theta_\delta$ in our notation and their $\theta_\beta$ is $-\theta_\beta$ in our notation.) See \cite{MR3969983} and references therein for related rigorous results for weakly self-avoiding walk and the $\varphi^4$ model on hierarchical lattices.  Note that \cref{thm:DM_Chi} and its strengthened form stated in \eqref{eq:DM_Chi_strong} imply that these exponents must satisfy $\theta_\beta \leq \theta_\delta \leq \theta_\gamma$ if they are well-defined.  Combining the predictions of \cite{essam1978percolation} with some simple heuristic scaling theory calculations yields the prediction
\begin{equation}
\nabla_{\beta_c-\eps} \asymp \left(\log \frac{1}{\eps} \right)^{8/7} \text{ as $\eps \downarrow 0$}
\end{equation}
for nearest-neighbour percolation on $\Z^6$, so that we should not expect the bounds on logarithmic corrections given by \cref{cor:logs} to be sharp in this example.

\begin{proof}[Proof of \cref{cor:logs} given \cref{thm:main_chi,thm:DM_Chi}]
We may assume without loss of generality that $\sum_{e\in E^\rightarrow_o}J_e=1$.
Let $\alpha \geq 0$, $C<\infty$, and $\delta$ be such that $\nabla_{\beta_c-\eps} \leq C(\log(1/\eps))^\alpha$ for every $0\leq \eps \leq \delta$. We may assume that $\delta\leq 1/2$ and will write $\asymp$, $\preceq$, and $\succeq$ for equalities and inequalities holding to within positive multiplicative constants depending only on $\alpha$, $C$, and $\delta$. We have by \cref{cor:integrated} that
\begin{equation}
\label{eq:Chilog}
\Chi_\beta  \preceq 
\left[\int_\beta^{\beta_c} \left(\log\frac{1}{\beta_c-\lambda}\right)^{-2\alpha} \dif \lambda\right]^{-1}
 \preceq |\beta-\beta_c|^{-1} \left(\log \frac{1}{|\beta-\beta_c|}\right)^{2\alpha}
\end{equation}
for every $ \beta_c-\delta \leq \beta < \beta_c$, where the latter inequality follows by considering the contribution to the integral from $\lambda \in [\beta,\beta+(\beta_c-\beta)/2]$.
Substituting \eqref{eq:Chilog} into \cref{thm:DM_Chi} with $\beta_1=\beta_c-n^{-1/2}$ and $\beta_2=\beta_c$ yields that
\[
\P_{\beta_c}(|K|\geq n) \preceq n^{-1} \Chi_{\beta_c-n^{-1/2}} \preceq n^{-1/2} \left(\log n \right)^{2\alpha}
\]
for every $n\geq 1$. Similarly, substituting \eqref{eq:Chilog} into \cref{thm:DM_Chi} with $\beta_1=\beta_c-\eps$ and $\beta_2=\beta_c+\eps$ yields that
\[
\P_{\beta_c+\eps}(|K|\geq n) \preceq \left(n^{-1}+\eps^2\right) \eps^{-1} \left(\log \frac{1}{\eps}\right)^{2\alpha}
\]
for every sufficiently small $\eps>0$ and every $n\geq 0$, so that the claimed inequality \eqref{logexponent:theta} follows by taking $n\uparrow \infty$. This completes the proof.
\end{proof}




\subsection{About the proof}

 The differential inequality of \cref{thm:main_chi} will be deduced from a more fundamental estimate involving two diagrammatic sums $A_\beta$ and $B_\beta$ that are more complicated than the usual triangle diagram. 
 Let $G$ be a countable, transitive weighted graph, and for each $\beta \geq 0$ consider the diagrammatic sums
\begin{align}
A_\beta &=  \sum_{v,w,x,y \in V} T_\beta(o,w)T_\beta(o,v)T_\beta(w,x)T_\beta(v,x) T_\beta(v,y) T_\beta(y,x) \label{eq:Adef}
\intertext{and}
B_\beta &=  \sum_{v,w,x,y \in V} T_\beta(o,w)T_\beta(o,v)T_\beta(w,x)T_\beta(v,x) T_\beta(w,y) T_\beta(v,y), \label{eq:Bdef}
\end{align}
both of which belong to $[1,\infty]$. See \cref{fig:diagrams_intro} for graphical representations of these sums. When $G$ is unimodular, the mass-transport principle \eqref{eq:MTP} allows us to exchange the roles of $o$ and $v$ to write these sums more succinctly as
\begin{align}
A_\beta &=  \sum_{v,w,x,y \in V} T_\beta(o,v)T_\beta(v,w)T_\beta(w,x)T_\beta(o,x) T_\beta(o,y) T_\beta(y,x) 
\nonumber\\
&\hspace{8.5cm}=\sum_{x\in V} T_\beta(o,x)T_\beta^2(o,x) T_\beta^3(o,x)
\label{eq:Adefunimod}
\intertext{and}
B_\beta &=  \sum_{v,w,x,y \in V} T_\beta(o,v)T_\beta(v,w)T_\beta(o,x)T_\beta(x,w) T_\beta(o,y) T_\beta(y,w)=\sum_{w \in V} T_\beta^2(o,w)^3 \label{eq:Bdefunimod}
\end{align}
for each $\beta \geq 0$. The diagram $B_\beta$ also arises in Hara and Slade's lace expansion analysis of percolation, see in particular \cite[Lemma 5.9]{MR1043524}.
%
\medskip

We will deduce \cref{thm:main_chi} 
as an immediate corollary of the following two propositions. 

\begin{figure}
\includegraphics[width=\textwidth]{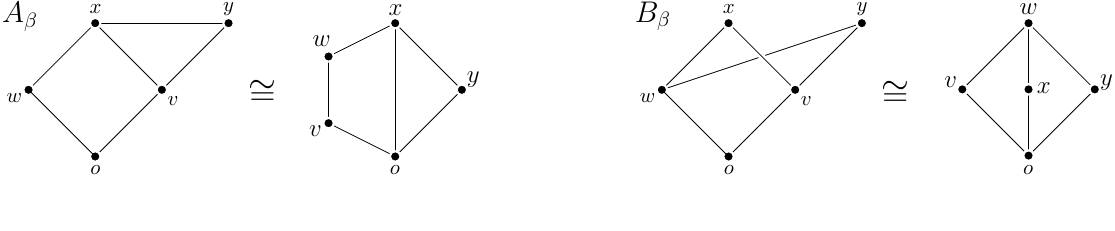}
\caption{The diagrammatic sums $A_\beta$ and $B_\beta$. In each case, we fix the origin $o$ and sum over every other vertex, with an edge of the diagram representing a copy of the two-point matrix $T_\beta$.
When $G$ is unimodular we can change which vertex of the diagram is pinned to the origin without changing the value of the corresponding sum. 
}
\label{fig:diagrams_intro}
\end{figure}

\begin{prop}
\label{prop:chi_diff_ineq}
Let $G=(V,E,J)$ be a connected, transitive weighted graph that is normalised so that $\sum_{e\in E^\rightarrow_o} J_e=1$. Then 
\begin{equation}
\label{eq:ABChi}
\frac{d \Chi_\beta}{d\beta} \geq \frac{\Chi_\beta(\Chi_\beta-\nabla_\beta)}{\beta^2(2A_\beta + B_\beta)}
\end{equation}
for every $0 < \beta < \beta_c$.
\end{prop}


\begin{prop}
\label{prop:triangle_comparison}
Let $G=(V,E,J)$ be a connected, unimodular, transitive weighted graph. Then 
$B_\beta \leq A_\beta \leq \nabla_\beta^{2}$ 
for every $\beta \geq 0$.
\end{prop}

Note that \cref{prop:chi_diff_ineq} does \emph{not} require unimodularity. 

\medskip

The inequality $A_\beta \leq \nabla_\beta^2$ of \cref{prop:triangle_comparison} follows easily from the fact that $T_\beta$ is positive definite \cite[Lemma 3.3]{MR762034} and hence that the maximal entries of $T_\beta^3$ lie on its diagonal. The inequality $B_\beta\leq A_\beta$ holds for rather more subtle reasons, but is also a consequence of $T_\beta$ being positive definite.
For Euclidean and hierarchical lattices this inequality admits a relatively straightforward proof by Fourier analysis, which is given at the end of \cref{sec:mainproof}. The general proof of \cref{prop:triangle_comparison} uses some interesting and (in our view) rather obscure facts about positive definite matrices; since this is somewhat tangential to the rest of the paper we defer the full proof, along with the exposition of the relevant linear-algebraic concepts, to \Cref{sec:linear_algebra}. We are optimistic that the techniques developed in this appendix may
 find further applications to probability theory and statistical physics in the future, perhaps to the problem of implementing the lace expansion \cite{MR2239599} in non-Euclidean settings.

\subsection{Applications to the hierarchical lattice}
\label{subsec:hierarchical_intro}

We now describe the applications of our results to long-range percolation on the hierarchical lattice. We begin by defining the  model.
Let $d\geq 1$, $L\geq 2$, and let $\mathbb{T}^d_L=(\Z/L\Z)^d$ be the discrete $d$-dimensional torus of side length $L$. 
We define the \textbf{hierarchical lattice} $\mathbbm{H}^d_L$ to be the countable abelian group $\bigoplus_{i=1}^\infty \mathbb{T}^d_L = \{x =(x_1,x_2,\ldots) \in (\mathbb{T}^d_L)^\N : x_i =0$ for all but finitely many $i\geq 0\}$ together with the group-invariant ultrametric
\[\|y-x\| = 
\begin{cases} 0 & x=y\\
L^{h(x,y)} & x \neq y
\end{cases} \qquad \text{ where }h(x,y)=\max\{i \geq 1: x_i \neq y_i\}.
\]
A function $J: \bbH^d_L \to [0,\infty)$ is said to be \textbf{radially symmetric} if $J(x)$ can be expressed as a function of $\|x\|$ and is said to be \textbf{integrable} if $\sum_{x\in \bbH^d_L} J(x)<\infty$. For example, writing $\langle x \rangle = 1 \vee \|x\|$ for each $x\in \bbH^d_L$, the function
$J(x)=\langle x\rangle^{-d-\alpha}$ is radially symmetric and integrable whenever $\alpha$ is positive. 
Given a radially symmetric, integrable function $J:\bbH^d_L\to [0,\infty)$, we define a transitive weighted graph with vertex set $\bbH^d_L$, edge set $\{\{x,y\}:x,y\in \bbH^d_L$, $ J(x-y)>0\}$, and weights given by $J(\{x,y\})=J(x-y)$; percolation on the resulting weighted graph is referred to as \emph{long-range percolation on the hierarchical lattice}. This model has the convenient feature that for each $n\geq 1$ the ultrametric ball $\Lambda_n:=\{x\in \bbH^d_L: \|x\|\leq L^n\}$ has the group structure of the torus $\bigoplus_{i=1}^n \mathbb{T}^d_L$ and induces a weighted subgraph of the hierarchical lattice that is itself transitive: there is no distinction between free and periodic boundary conditions in hierarchical models.

\medskip

The principal result of our earlier work \cite{hutchcroft2021critical} states that if $J:\bbH^d_L\to[0,\infty)$ is an integrable, radially symmetric function satisfying $c\|x\|^{-d-\alpha} \leq J(x)\leq C\|x\|^{-d-\alpha}$ for every $x\in \bbH^d_L\setminus \{0\}$ for some $0<\alpha <d$ and some positive constants $c$ and $C$ then there exist positive constants $a$ and $A$ such that
\begin{equation}
\label{eq:hierarchical_two_point}
a \langle x-y\rangle^{-d+\alpha} \leq \P_{\beta_c}(x\leftrightarrow y) \leq A \langle x-y \rangle^{-d+\alpha}
\end{equation}
for every $x,y\in \bbH^d_L$. It follows from this result that the model satisfies the triangle condition if and only if $\alpha < d/3$ \cite[Corollary 1.4]{hutchcroft2021critical}, and indeed it is proven in \cite[Corollary 1.5]{hutchcroft2021critical} that the model does \emph{not} have mean-field critical exponents when $\alpha >d/3$. Note that the simple expression for the two-point function \eqref{eq:hierarchical_two_point} therefore holds both inside and outside of the mean-field regime.

\medskip

 The techniques we develop in this paper can be used to prove that mean-field critical behaviour holds to within polylogarithmic factors in the upper-critical case $\alpha=d/3$.

\begin{thm}
\label{thm:d/3}
Let $J:\mathbb{H}^d_L \to [0,\infty)$ be a radially symmetric, integrable function, let $0<\alpha<d$, and suppose that there exist constants $c$ and $C$ such that $c \|x\|^{-d-\alpha} \leq J(x) \leq C\|x\|^{-d-\alpha}$ for every $x\in \mathbb{H}^d_L \setminus \{0\}$. If $\alpha=d/3$ then
\begingroup
\begin{align}
\Chi_{\beta_c-\eps} &\preceq \Bigl(\log \frac{1}{\eps}\Bigr)^{2}\cdot\frac{1}{\eps} &\text{ as }\eps &\downarrow 0,
\\
\P_{\beta_c}\left( |K| \geq n\right) &\preceq \hspace{0.13em} (\hspace{0.2em}\log n\hspace{0.15em})^{2}  \cdot \frac{1}{n^{1/2}} &\text{ as }n&\uparrow \infty,  \text{ and}
\\
\P_{\beta_c+\eps}\left(|K|=\infty\right) &\preceq  \Bigl(\log \frac{1}{\eps}\Bigr)^{2}\cdot \eps&\text{ as }\eps &\downarrow 0.
\end{align}
\endgroup
In particular, mean-field critical behaviour holds to within polylogarithmic factors.
\end{thm}

We are not aware of heuristic work giving predictions for the true exponents in the logarithmic corrections appearing here, but do not expect the bounds given by \cref{thm:d/3} to be sharp.
Note that \cref{thm:d/3} is not an instance of \cref{cor:logs} since the results of \cite{hutchcroft2021critical} do not \emph{a priori} give any control of the rate of growth of the near-critical triangle $\nabla_{\beta_c-\eps}$ when $\alpha = d/3$. Related results stating roughly that the exponents take values close to their mean-field values when $\alpha$ is slightly larger than $d/3$ are given in \cref{thm:>d/3}.

\section{Derivation of the main differential inequality}
\label{sec:mainproof}

In this section we prove \cref{prop:chi_diff_ineq} and the Euclidean and hierarchical cases of \cref{prop:triangle_comparison}, yielding a proof of \cref{thm:main_chi} in this case. The complete proof of \cref{thm:main_chi} for general transitive weighted graphs is deferred to \Cref{sec:linear_algebra}. 

\medskip

We begin with \cref{prop:chi_diff_ineq}. 
Our proof uses \textbf{Russo's formula} \cite[Chapter 2.4]{grimmett2010percolation}, which states in our context that if $G=(V,E,J)$ is a countable weighted graph and $A \subseteq \{0,1\}^E$ is an increasing event depending on at most finitely many edges then
\[
\frac{d}{d\beta} \P_\beta(A) = \sum_{e\in E} J_e e^{-\beta J_e} \P_\beta(e \text{ is pivotal for $A$}) = \sum_{e\in E} J_e  \P_\beta(e \text{ is closed pivotal for $A$}).
\]
Here we recall that a set $A\subseteq \{0,1\}^E$ is said to be \textbf{increasing} if $(\omega \in A)\Rightarrow (\omega'\in A)$ for every $\omega,\omega' \in \{0,1\}^E$ such that $\omega'(e) \geq \omega(e)$ for every $e\in E$, and that given a percolation configuration $\omega$ and an increasing event $A$, an edge $e\in E$ is said to be \textbf{pivotal} if $\omega \cup \{e\} \in A$ and $\omega \setminus \{e\} \notin A$; we say that an edge is \textbf{closed pivotal} for $A$ if it is closed and pivotal for $A$. More generally \cite[Equation 2.28]{grimmett2010percolation}, if $A$ is an increasing event that may depend on infinitely many edges, we have the inequality
\[
\lrDini{\beta} \P_\beta(A) \geq \sum_{e\in E} J_e e^{-\beta J_e} \P_\beta(e \text{ is pivotal for $A$}) = \sum_{e\in E} J_e  \P_\beta(e \text{ is closed pivotal for $A$})
\]
where $\lrDini{\beta} \P_\beta(A) = \liminf_{\eps\downarrow 0} \frac{1}{\eps}\left(\P_{\beta+\eps}(A)-\P_\beta(A)\right)$ is the \textbf{lower-right Dini derivative} of $\P_\beta(A)$. 

\medskip

Let $G=(V,E,J)$ be a countable weighted graph, and let $S$ be a finite set of vertices of $G$. 
Let $\partial_E^\rightarrow S$ denote the set of oriented edges of $G$ with $e^- \in S$ and $e^+ \notin S$. Given $u,v\in V \setminus S$, we write $\{u\leftrightarrow v$ off $S\}$ to mean that there exists an open path connecting $u$ and $v$ that does not visit any vertex of $S$.
For each $\beta \geq 0$ let $\Phi_\beta(S)$ be defined by
\begin{align*}
\Phi_\beta(S) &=   \sum_{e \in \partial_E^\rightarrow S} \sum_{v\in V \setminus S} J_e \P_{\beta}(e^+ \leftrightarrow v \text{ off $S$}).
\end{align*}
Note that an edge $e$ is a closed pivotal for the event $\{o\leftrightarrow v\}$ if and only if there exists a (necessarily unique) orientation of $e$ such that $e^- \in K$, $e^+ \notin K$, and $v$ is connected to $e^+$ off of $K$. As such, Russo's formula implies that
\begin{equation}
\lrDini{\beta} \E_\beta |K_v| \geq \E_\beta \left[\Phi_\beta(K_v)\right] 
\end{equation}
for every $\beta \geq 0$ and $v\in V$, where $K_v$ denotes the cluster of $v$.

\medskip

We begin our analysis with the following key lemma, which does not require transitivity. We write $\Chi_\beta^\mathrm{min}=\inf_{v\in V} \E_\beta |K_v|$ and  $\Chi_\beta^\mathrm{max}=\sup_{v\in V} \E_\beta |K_v|$ so that $\Chi_\beta^\mathrm{min} = \Chi_\beta^\mathrm{max} = \Chi_\beta$ when $G$ is transitive.

\begin{lemma}
\label{lem:Psi_general}
Let $G$ be a countable weighted graph normalized so that $\sup_{v\in V} \sum_{e\in E^\rightarrow_v} J_e \leq 1$. Then
\begin{align}
\label{eq:Phi_general}
\Phi_\beta(S) &\geq 
\frac{\bigl(\Chi^\mathrm{min}_\beta |S| -\sum_{u,v \in S} T_\beta(u,v) \bigr)^2 }{\beta^2 \Chi^\mathrm{max}_\beta \sum_{u,v,w \in S} T_\beta(u,v) T_\beta(u,w)}
\end{align}
for every $\beta>0$ and every finite set of vertices $S$ in $G$.
\end{lemma}

The proof of this lemma will apply the \emph{van den Berg and Kesten} (BK) \emph{inequality} \cite{MR799280}  and the attendant notion of the \emph{disjoint occurrence} $A\circ B$ of two sets $A$ and $B$; we refer the unfamiliar reader to \cite[Chapter 2.3]{grimmett2010percolation} for background.


\begin{proof}[Proof of \cref{lem:Psi_general}]
Fix $S \subseteq V$ and $\beta >0$. 
  Observe that for each $u\in S$ and $w\notin S$ we have that
\[
\{u \leftrightarrow w \} \subseteq  \bigcup_{e \in \partial_E^\rightarrow S} \{u \leftrightarrow e^-\} \circ \{ e \text{ open} \} 
\circ
\{ e^+ \leftrightarrow w \text{ off $S$}\}.
\]
Indeed, if $\gamma$ is a simple open path from $u$ to $w$ and $e$ is the oriented edge that is crossed by $\gamma$ as it leaves $S$ for the last time, then the pieces of $\gamma$ before and after it crosses $e$ are disjoint witnesses for $\{u \leftrightarrow e^-\}$ and $\{e^+ \to w $ off $S\}$ that are both disjoint from $e$. 
Thus, applying a union bound and the BK inequality yields that
\begin{align*}
T_\beta(u,w) &\leq  \sum_{v\in S} T_\beta(u,v) \sum_{e \in E^\rightarrow_v} \mathbbm{1}(e^+ \notin S) (1-e^{-\beta J_e})\P_{\beta}(e^+ \leftrightarrow  w  \text{ off $S$})\\
&\leq  \beta \sum_{v\in S} T_\beta(u,v) \sum_{e \in E^\rightarrow_v} \mathbbm{1}(e^+ \notin S) J_e \P_{\beta}(e^+ \leftrightarrow  w  \text{ off $S$})
\end{align*}
for every $u \in S$ and $w\in V\setminus S$, where we used the inequality $1-e^{-t} \leq t$ in the second line. Summing over $u \in S$ and $w\in V \setminus S$, we deduce that
\[
\sum_{w\in V \setminus S} \sum_{u \in S} T_\beta(u,w) \leq \beta \sum_{w\in V \setminus S} \sum_{u,v\in S} T_\beta(u,v) \sum_{e \in E^\rightarrow_v} \mathbbm{1}(e^+ \notin S)J_e \P_{\beta}(e^+ \leftrightarrow  w \text{ off $S$}).
\]
 Let $f: V \to \R$ 
 be defined by 
\[
f(v) = \mathbbm{1}(v\in S) \sum_{w\in V \setminus S}\sum_{e \in E^\rightarrow_v} \mathbbm{1}(e^+ \notin S) J_e \P_{\beta}(e^+ \leftrightarrow  w  \text{ off $S$})
\]
so that the above inequality may be rewritten  as 
\[\sum_{w\in V \setminus S} \sum_{u \in S} T_\beta(u,w) \leq\beta \langle  \mathbbm{1}_S T_\beta \mathbbm{1}_S, f\rangle,
\]
where $\langle f,g \rangle = \sum_{v\in V} f(v)g(v)$ is the standard inner product on $L^2(V)$ and $\mathbbm{1}_S T_\beta \mathbbm{1}_S (v)=\mathbbm{1}(v\in S) \sum_{u\in S} T_\beta(v,u)$ for every $v\in V$. Applying Cauchy-Schwarz and (a trivial special case of) H\"older's inequality, we obtain that
\begin{equation}
\label{eq:Holder}
\sum_{w\in V \setminus S} \sum_{u \in S} T_\beta(u,w)  \leq \beta \langle \mathbbm{1}_S T_\beta \mathbbm{1}_S, f \rangle \leq \beta \|\mathbbm{1}_S T_\beta \mathbbm{1}_S \|_2 \| f \|_2 \leq \beta \|\mathbbm{1}_S T_\beta \mathbbm{1}_S \|_2 \| f \|_1^{1/2} \|f\|_\infty^{1/2}.
\end{equation}
We have from the definitions that $\|\mathbbm{1}_S T_\beta \mathbbm{1}_S\|_2^2 = \sum_{u,v,w\in S} T_\beta(u,v)T_\beta(u,w)$, $\|f\|_1=\Phi_{\beta}(S)$, 
and
\begin{align*}
\sum_{w\in V\setminus S}\sum_{u \in S} T_\beta(u,w) &= 
\sum_{w\in V}\sum_{u \in S} T_\beta(u,w) -  \sum_{u,v \in S} T_\beta(u,w)
\geq 
\Chi_{\beta}^\mathrm{min} |S| - \sum_{u,v \in S}T_\beta(u,v). 
\end{align*}
Since $\sum_{e\in E^\rightarrow_v} J_e \leq 1$ for every $v\in V$ we also have that
\[\|f\|_\infty \leq \sup_{v\in V} \sum_{w\in V} \sum_{e\in E^\rightarrow_v} J_e \P_\beta(e^+ \leftrightarrow w) \leq \Chi_{\beta}^\mathrm{max},\]
and the claim follows by substituting these four estimates into \eqref{eq:Holder} and rearranging.
\end{proof}

We now deduce \cref{prop:chi_diff_ineq} from \cref{lem:Psi_general}. 
Before beginning the proof, let us recall the following trivial consequence of Cauchy-Schwarz inequality: If $X$ and $Y$ are real-valued random variables defined on the same probability space such that $Y$ is positive almost surely and $\E[ Y ]$ is finite then
\[
\E \!\left[\frac{X^2}{Y} \right] \geq \myfrac{\E [ X ]^2}{\E[ Y ]^{\phantom{2}}}.
\]

\begin{proof}[Proof of \cref{prop:chi_diff_ineq}]
Recall that $K$ denotes the cluster of the origin. We have by Russo's formula and \eqref{eq:Phi_general} that
\begin{align*}
\frac{d \Chi_\beta}{d\beta} &\geq \E_p \left[\Phi_\beta\left(K\right)\right] \geq \E_\beta\left[\frac{\bigl(\Chi_\beta|K|-\sum_{x,y \in K} T_\beta(x,y)\bigr)^2}{\beta^2 \Chi_\beta \sum_{a,b,c \in K} T_\beta(a,b)T_\beta(a,c)}\right]
\end{align*}
for every $0<\beta<\beta_c$. (As previously mentioned, $\Chi_\beta$ is a smooth function of $\beta$ on $[0,\beta_c)$ \cite[Chapter 6.4]{grimmett2010percolation}, so that its usual derivative and lower-right Dini derivative coincide on this interval.)
Applying the Cauchy-Schwarz inequality as above with the random variables $X=\Chi_\beta|K|-\sum_{x,y \in K} T_\beta(x,y)$ and  $Y=\sum_{a,b,c \in K} T_\beta(a,b)T_\beta(a,c)$  yields that
\begin{align}
\label{eq:diagram_derivation0}
\frac{d \Chi_\beta}{d\beta} \geq \frac{1}{\beta^2 \Chi_\beta} \E_\beta\left[\Chi_\beta|K|-\sum_{x,y \in K} T_\beta(x,y)\right]^2 \E_\beta\left[ \sum_{a,b,c \in K} T_\beta(a,b)T_\beta(a,c)\right]^{-1}
\end{align}
for every $0<\beta<\beta_c$. (Note that the random variable $\sum_{a,b,c \in K} T_\beta(a,b)T_\beta(a,c)$ is bounded by $|K|^3$ and is therefore integrable for $\beta<\beta_c$ by sharpness of the phase transition \cite{duminil2015new,1901.10363,aizenman1987sharpness}.)
Thus, to complete the proof it suffices to prove that 
\begin{equation}
\label{eq:diagram_derivation1}
 \E_{\beta}\left[\sum_{x,y \in K} T_\beta(x,y)\right] \leq \Chi_{\beta} \nabla_\beta \quad \text{ and } \quad 
 \E_{\beta}\left[ \sum_{a,b,c \in K} T_\beta(a,b)T_\beta(a,c)\right] \leq (2A_\beta + B_\beta)\Chi_{\beta}
\end{equation}
for every $0<\beta<\beta_c$; the claimed estimate \eqref{eq:ABChi} will then follow by substituting \eqref{eq:diagram_derivation1} into \eqref{eq:diagram_derivation0}. 

\begin{figure}[t!]
\centering
\includegraphics[width=0.8\textwidth]{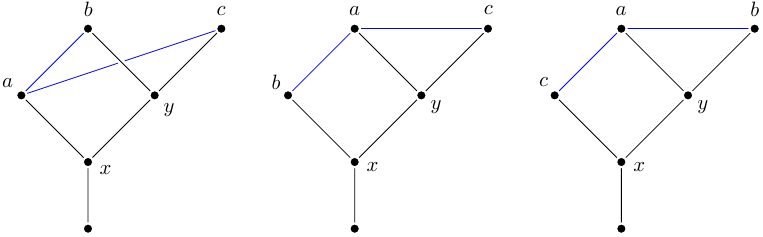}
\caption{Diagrammatic representations of the three sums contributing to the upper bound on the expectation of $\E_\beta \sum_{a,b,c\in K} T_\beta(a,b)T_\beta(a,c)$. Fixing $a$, $b$, and $c$, weighting each black edge by $T_\beta$, and summing over $x$ and $y$ gives the tree-graph bound on the probability that $a$, $b$, and $c$ all belong to the cluster of the origin.}
\label{fig:diagrams_tree_graph}
\end{figure}

\medskip

Both inequalities of \eqref{eq:diagram_derivation1} will follow by the same reasoning used to prove the ``tree graph inequalities" of Aizenman and Newman \cite{MR762034}, an account of which can also be found in \cite[Chapter 6.3]{grimmett2010percolation}.
We begin with the first of the two inequalities claimed in \eqref{eq:diagram_derivation1}. Let $x,y \in V$, and suppose that $x$ and $y$ are both connected to $o$. We claim that there must exist a vertex $w \in V$ (possibly equal to one of $o,x,$ or $y$) such that the event $\{o \leftrightarrow w\} \circ \{w \leftrightarrow x\} \circ \{w \leftrightarrow y\}$ holds. Indeed, if $\gamma_1$ is an open simple path from $v$ to $x$ and $\gamma_2$ is an open simple path from $v$ to $y$, then the last vertex of $\gamma_1$ visited by $\gamma_2$ has this property. Thus, we have by the BK inequality and the union bound that
\begin{equation}
\P_{\beta}(x,y \in K) \leq \sum_{w\in V}  T_\beta(o,w)T_\beta(w,x)T_\beta(w,y)
\label{eq:threepoint}
\end{equation}
for every $v,x,y \in V$. It follows that
\begin{multline*}
 \E_{\beta}\left[\sum_{x,y \in K} T_\beta(x,y)\right] =  \sum_{x,y \in V} \P_{\beta}(x,y \in K) T_\beta(x,y) 
\\\leq \sum_{x,y,w\in V} T_\beta(o,w)T_\beta(w,x)T_\beta(w,y) T_\beta(x,y)=\Chi_{\beta}\nabla_\beta
\end{multline*}
as claimed. We now turn to the second inequality of \eqref{eq:diagram_derivation1}, whose proof is very similar although the required expressions are larger.
Let $a,b,c \in V$. It follows by similar reasoning to the two-vertex case above that if $a$, $b$, and $c$ are all connected to $o$  then there exist vertices $x$ and $y$ such that at least one of the events
\begin{align*}
&\{o \leftrightarrow x\} \circ \{x \leftrightarrow a\} \circ \{ x \leftrightarrow y\} \circ \{y \leftrightarrow b\} \circ \{y \leftrightarrow c\},\\
&\{o \leftrightarrow x\} \circ \{x \leftrightarrow b\} \circ \{ x \leftrightarrow y\} \circ \{y \leftrightarrow a\} \circ \{y \leftrightarrow c\}, \text{ or}\\
&\{o \leftrightarrow x\} \circ \{x \leftrightarrow c\} \circ \{ x \leftrightarrow y\} \circ \{y \leftrightarrow a\} \circ \{y \leftrightarrow b\}
\end{align*}
holds. See \cite[Equation 6.93]{grimmett2010percolation} and its proof for details. Thus, it follows from a union bound and the BK inequality that
\begin{align}\P_{\beta}(a,b,c \in K)
&\leq \sum_{x,y \in V} T_{\beta}(o,x)T_\beta(x,a)T_\beta(x,y)T_\beta(y,b)T_\beta(y,c)
\nonumber\\
&\hspace{2cm}+\sum_{x,y \in V} T_{\beta}(o,x)T_\beta(x,b)T_\beta(x,y)T_\beta(y,a)T_\beta(y,c)
\nonumber\\
&\hspace{4cm}+\sum_{x,y \in V} T_{\beta}(o,x)T_\beta(x,c)T_\beta(x,y)T_\beta(y,a)T_\beta(y,b)
\label{eq:fourpoint}
\end{align}
for every $a,b,c\in V$ and $\beta \geq 0$.
Summing over $a,b,c\in V$, we deduce from this together with the definitions of $A_\beta$ and $B_\beta$ that
\[
\E_{\beta}\left[ \sum_{a,b,c \in K} T_\beta(a,b)T_\beta(a,c)\right]  = \sum_{a,b,c\in V}\P_{\beta}(a,b,c \in K) T_\beta(a,b)T_\beta(a,c) \leq \Chi_{\beta} (2A_\beta+B_\beta)
\]
for every $0\leq \beta < \beta_c$ as claimed; see \cref{fig:diagrams_tree_graph} for graphical representations of the resulting diagrammatic sums. \qedhere

\end{proof}








We finish this section by proving \cref{prop:triangle_comparison} in the Euclidean and hierarchical cases using Fourier analysis; the same proof is available on any abelian Cayley graph, but breaks down for general transitive graphs (including Cayley graphs of nonabelian groups), where Fourier analysis is not available. The proof of the general case of the proposition is deferred to \Cref{sec:linear_algebra}.

\begin{proof}[Proof of \cref{prop:triangle_comparison} in the Euclidean and hierarchical cases via Fourier analysis] 
We begin with the case that $G=\Z^d$ and $J$ is translation-invariant. 
Let $0\leq \beta<\beta_c$ and let $\mathbb{T}^d=[-\pi,\pi]^d$ be the $d$-dimensional torus.
Let $\tau_\beta(x)=T_\beta(0,x)$, which belongs to $\ell^1(\Z^d)$ by sharpness of the phase transition, and let $\hat \tau_\beta: \mathbb{T}^d \to \C$ be its Fourier transform
\[
\hat \tau_\beta(\theta) = \sum_{x\in \Z^d} \tau_\beta(x) e^{i \theta \cdot x}.
\]
The symmetry $\tau_\beta(x)=\tau_\beta(-x)$ ensures that $\hat \tau_\beta (\theta) \in \R$ for every $\theta \in \mathbb{T}^d$, while the fact that $\tau_\beta$ is real-valued implies that $\hat \tau_\beta(\theta)=\hat \tau_\beta(-\theta)$ for every $\theta \in \mathbb{T}^d$. Moreover, 
it is a lemma due to Aizenman and Newman \cite[Lemma 3.3]{MR762034} (see also \cref{lem:positive_definite} below) that the matrix $T_\beta$ is positive semidefinite and hence that  $\hat \tau_\beta(\theta)$ is  non-negative for every $\theta \in \mathbb{T}^d$.

\medskip

Letting $*$ denote convolution (on either $\Z^d$ or $\mathbb{T}^d$, taken with respect to either counting measure or normalized Lebesgue measure as appropriate) and using that Fourier transforms exchange the roles of convolution and multiplication,  we can  write the diagrammatic sums $A_\beta$ and $B_\beta$ as
\begin{align}
A_\beta &= \sum_{x\in \Z^d} \tau_\beta(x) [\tau_\beta * \tau_\beta](x) [\tau_\beta* \tau_\beta * \tau_\beta] (x)
= [\hat\tau_\beta * \hat\tau_\beta^2 * \hat\tau_\beta^3] (0)
\label{eq:FourierA}
\end{align}
and
\begin{align}
B_\beta = \sum_{x\in \Z^d} [\tau_\beta * \tau_\beta](x)^3 = \hat\tau_\beta^2 * \hat\tau_\beta^2 * \hat\tau_\beta^2 (0),
\label{eq:FourierB}
\end{align}
where we used that $\sum_{x \in \Z^d} f(x) = \hat f(0)$ for every $f \in \ell^1(\Z^d)$. Similarly, we can express $\nabla_\beta$ via either of the two equivalent expressions
\[
\nabla_\beta = \sum_{x\in \Z^d} \tau_\beta(x)[\tau_\beta * \tau_\beta](x) = [\hat \tau_\beta * \hat \tau_\beta^2](0) \quad \text{ and } \quad \nabla_\beta =  [\tau_\beta*\tau_\beta * \tau_\beta](0)  = \frac{1}{(2\pi)^d}\int_{\mathbb{T}^d} \hat \tau_\beta^3(\theta) \dif \theta.
\]
%
%
This Fourier-analytic perspective can easily be used to prove the bound $A_\beta \leq \nabla_\beta^2$. (In fact this inequality is seen even more easily directly in physical space as in \cref{lem:Abound}; we give a Fourier proof here to set the stage for the more difficult inequality $B_\beta \leq A_\beta$ below.) Indeed, since $\hat \tau_\beta$ is non-negative we may apply H\"older's inequality to deduce that
\begin{align*}
\hat \tau_\beta * \hat \tau_\beta^2 (\theta) &= \frac{1}{(2\pi)^d}\int_{\mathbb{T}^d} \hat \tau_\beta(\phi)\hat \tau_\beta^2 (\theta-\phi) \dif \phi
\leq \frac{1}{(2\pi)^d}\left[\int_{\mathbb{T}^d} \hat \tau_\beta(\phi)^3 \dif \phi \right]^{1/3}\left[\int_{\mathbb{T}^d} \hat \tau_\beta(\theta-\phi)^{3} \dif \phi \right]^{2/3}
\\
&=\frac{1}{(2\pi)^d}\int_{\mathbb{T}^d} \hat \tau_\beta(\phi)^3 \dif \phi = \frac{1}{(2\pi)^d}\int_{\mathbb{T}^d} \hat \tau_\beta(\phi)\hat \tau_\beta^2 (-\phi) \dif \phi = \hat \tau_\beta * \hat \tau_\beta^2 (0)
\end{align*}
for every $\theta \in \mathbb{T}^d$, where we used that $\hat \tau_\beta$ is even in the second equality on the second line, and hence that
\[
A_\beta = \frac{1}{(2\pi)^d}\int_{\mathbb{T}^d} [\hat \tau_\beta * \hat \tau_\beta^2](\phi) \hat \tau_\beta^3(-\phi) \dif \phi \leq [\hat \tau_\beta * \hat \tau_\beta^2](0) \frac{1}{(2\pi)^d}\int_{\mathbb{T}^d} \hat \tau_\beta^3(-\phi) \dif \phi = \nabla_\beta^2
\]
as claimed, where we used the non-negativity of $\hat \tau_\beta$ in the central inequality.

\medskip

It remains to establish the less obvious bound $B_\beta \leq A_\beta$.
To this end, we claim that if $f: \mathbb{T}^d \to [0,\infty)$ is an arbitrary non-negative function then 
\begin{equation}
\label{eq:FourierAndo}
f^2*f^2 (\theta) \leq f * f^3(\theta)\end{equation} for every $\theta \in \mathbb{T}^d$. Indeed, if $f,g: \mathbb{T}^d \to [0,\infty)$ are any two functions then the AM-GM inequality implies that
\begin{align*}
[f^2 * g^2](\theta) &= \frac{1}{(2\pi)^d}\int_{\mathbb{T}^d} f^2(\phi)g^2(\theta-\phi) \dif \phi \\&\leq  \frac{1}{(2\pi)^d}\int_{\mathbb{T}^d} \frac{1}{2}\left[f(\phi)g^3(\theta-\phi)+f^3(\phi)g(\theta-\phi)\right] \dif \phi
= \frac{1}{2}\left([f*g^3](\theta)+[f^3*g](\theta) \right)
\end{align*}
for every $\theta \in \mathbb{T}^d$, 
and the claim follows by taking $f=g$. Applying the inequality \eqref{eq:FourierAndo} with $f=\hat \tau_\beta$ and using positivity of $\hat \tau_\beta^2$ we deduce that $\hat \tau_\beta^2 * \hat \tau_\beta^2 * \hat \tau^2_\beta (\theta) \leq \hat \tau_\beta * \hat \tau_\beta^3 * \hat \tau^2_\beta (\theta)$ for every $\theta \in \mathbb{T}^d$ and consequently that $B_\beta \leq A_\beta$ as claimed. This completes the proof in the case $G=\Z^d$.

\medskip

The proof for the hierarchical lattice follows similarly, but where the Fourier transform $\hat f : \bbH^d_L\to \C$ of a function $f\in \ell^1(\bbH^d_L,\C)$ is defined by 
\[
\hat f(y) = \sum_{x\in \bbH^d_L} f(x) e^{i y \cdot x} \qquad \text{ where $y\cdot x=\frac{2\pi}{L} \sum_{i=0}^\infty y_i \cdot x_i$ for every $x,y\in \bbH^d_L$.}
\]
Note in particular that the hierarchical lattice is its own Pontryagin dual. The Fourier transform on $\bbH^d_L$ enjoys all the same properties that are used to carry out our analysis on $\Z^d$ above; we omit the details.
\end{proof}

\section{The tail of the volume and the infinite cluster density}



In this section we show that mean-field behaviour of the susceptibility implies mean-field behaviour of the critical volume distribution and the infinite cluster density.
We begin by stating and proving a generalization to weighted graphs of an inequality established in the very recent work of Dewan and Muirhead \cite[Proposition 2.10]{dewan2021upper}, from which we will deduce \cref{thm:DM_Chi} as a corollary. The statement of this inequality will involve the notion of \emph{decision trees}, which have recently come to play an important role in mathematical physics following the breakthrough work of Duminil-Copin, Raoufi, and Tassion \cite{MR3898174}.

\medskip
Let $E$ be a countable set and let $E^*=\bigcup_{n \in \N \cup \{\infty\}} E^n$ be the set of finite or infinite sequences in $E$, which is equipped with the product topology and associated Borel $\sigma$-algebra.
  A \textbf{decision tree} is a function 
  $T:\{0,1\}^E \to E^*$ such that
$T_1(\omega) = e_1$ for some  fixed $e_1 \in E$ and for each $n \geq 2$ there exists a function $S_n : (E\times \{0,1\})^{n-1} \to E \cup \varnothing$ such that either $S_n\left[\left(T_i,\omega(T_i)\right)_{i=1}^{n-1}\right]=\varnothing$ in which case $T(\omega)=(T_1(\omega),\ldots,T_{n-1}(\omega))$ (i.e., the decision tree halts) or else
\[
T_n(\omega) = S_n\left[\left(T_i,\omega(T_i)\right)_{i=1}^{n-1}\right].
\]
That is, $T$ is a deterministic procedure for querying the values of the configuration $\omega \in \{0,1\}^E$ that starts by querying the value of some fixed edge $e_1$ and at each subsequent step chooses either to halt or to query the value of some other edge as a function of the values it has already observed. 

\medskip

Given a decision tree $T$ and $\omega\in \{0,1\}^E$, we write $\tau(\omega)$ for the length of the sequence $T(\omega)$.
Given a Borel subset $A$ of $\{0,1\}^E$, we say that a decision tree $T$ \textbf{Borel-computes} $A$ if there exists a Borel subset $\sA$ of $\{0,1\}^*$ such that $\omega \in A$ if and only if $(\omega(T_i))_{i=1}^{\tau(\omega)} \in \sA$. Given a decision tree $T$, a measure $\mu$ on $\{0,1\}^E$, and $e\in E$, we define the \textbf{revealment probability}
\[
\operatorname{Rev}(\mu,T,e) = \mu(\{ \omega: \exists n\geq 1 \text{ such that } T_n(\omega)=e \}).
\]
Given a countable weighted graph $G=(V,E,J)$, $\beta \geq 0$, and a decision tree $T:\{0,1\}^E \to E^*$, we write $\operatorname{Rev}_\beta(T,e)=\operatorname{Rev}(\P_\beta,T,e)$. 

\begin{theorem}[Generalised Dewan-Muirhead]
\label{thm:DewanMuirhead}
Let $G=(V,E,J)$ be a countable weighted graph, let $A \subseteq \{0,1\}^E$ be a Borel set, and let $T$ be a decision tree that Borel-computes the event $A$. Then 
\begin{equation}
|\P_{\beta_1}(A)-\P_{\beta_2}(A)|^2 \leq \frac{|\beta_1-\beta_2|^2}{\min\{\beta_1,\beta_2\}}  \max\{\P_{\beta_1}(A),\P_{\beta_2}(A)\} \sum_{e\in E} J_e \operatorname{Rev}_{\beta_1}(T,e)
\end{equation}
for every $\beta_1,\beta_2 \geq 0$.
\end{theorem}

This inequality plays an interesting complementary role to the \emph{OSSS inequality} of O'Donnel, Saks, Schramm, and Servedio \cite{o2005every}, which roughly states that if $A$ is \emph{increasing} then the existence of a decision tree computing $A$  with low \emph{maximum} revealment implies that the logarithmic derivative of $\P_\beta(A)$ is \emph{large}. 

\begin{remark}
The original Dewan-Muirhead inequality required $A$ to depend on at most finitely many edges. This assumption is not appropriate in the long-range case, and we have introduced the notion of a decision tree Borel-computing an event to circumvent this issue. In other parts of the literature, one often defines a decision tree $T$ to compute an event $A$ if $A$ belongs to the \emph{completion} of the $\sigma$-algebra generated by $T$; considering infinite-volume cases of \cref{thm:DewanMuirhead} in which $\P_{\beta_1}(A)=1$ and $\P_{\beta_2}(A)=0$ -- in which case $A$ is trivially in the completion of the $\sigma$-algebra generated by \emph{any} decision tree -- shows that this notion of computation cannot be used in the statement of \cref{thm:DewanMuirhead}.
\end{remark}

The proof of \cref{thm:DewanMuirhead} will rely on the notion of \emph{relative entropy}. Given two probability measures $\mu$ and $\nu$ on a common measurable space $\Omega$, recall that the \textbf{relative entropy} (a.k.a.\ \textbf{Kullback-Liebler divergence}) from $\nu$ to $\mu$ is defined to be 
\[
D_{\mathrm{KL}}(\mu || \nu) =  \int \log \left(\frac{d\mu}{d\nu} \right) d\mu(x) 
\]
if $\mu$ is absolutely continuous with respect to $\nu$, and is defined to be infinite otherwise.
 Given random variables $X$ and $Y$ taking values in the same space (but not necessarily defined on the same probability space), we write $D_\mathrm{KL}(X||Y)$ for the relative entropy from the law of $Y$ to the law of $X$.
The relevance of this quantity to \cref{thm:DewanMuirhead} arises from a generalization of Pinsker's inequality established in \cite[Lemma 2.12]{dewan2021upper}, which states that if $\mu$ and $\nu$ are two probability measures on a common measurable space $\Omega$ then
\begin{equation}
\label{eq:Pinsker}
|\mu(A)-\nu(A)|^2 \leq 2 D_\mathrm{KL}(\mu||\nu) \max\{\mu(A),\nu(A)\}
\end{equation}
for every event $A \subseteq \Omega$. Indeed, the expression $\frac{|\beta_1-\beta_2|^2}{2\min\{\beta_1,\beta_2\}}  \sum_{e\in E} J_e \operatorname{Rev}_{\beta_1}(T,e)$ appearing in \cref{thm:DewanMuirhead} will arise as an upper bound on the relative entropy of two appropriately defined random variables. We will also use the \textbf{chain rule} for the relative entropy, which states that if $(X_1,X_2)$ and $(Y_1,Y_2)$ are random variables taking values in the same product space $\Omega_1 \times \Omega_2$ then
\begin{equation}
\label{eq:chainrule}
D_\mathrm{KL}\Bigl((X_1,X_2) || (Y_1,Y_2)\Bigr) = D_\mathrm{KL}(X_1||Y_1) + \E_{x\sim X_1} \left[ D_\mathrm{KL}\Bigl( (X_2 | X_1 =x ) \; \big|\big| \; (Y_2 | Y_1=x) \Bigr) \right],
\end{equation}
where we write $(X_2 | X_1 =x )$ for the conditional distribution of $X_2$ given $X_1=x$ and write $\E_{x\sim X_1}$ for an expectation taken over a random $x$ with the law of $X_1$.

\medskip

We now begin to work towards the proof of \cref{thm:DewanMuirhead}.
We will use the following estimate on the relative entropy of Bernoulli random variables, which we express in a form that is convenient for our applications to percolation on weighted graphs.


\begin{lemma}
\label{lem:Bernoulli_DKL}
For each $p\in [0,1]$ let $\operatorname{Ber}(p)$ denote the law of a Bernoulli-$p$ random variable. The estimate
\begin{equation}
\label{eq:Bernoulli_DKL}
D_\mathrm{KL}\Bigl(\operatorname{Ber}\bigl(1-e^{-a}\bigr)\Big|\Big|\operatorname{Ber}\bigl(1-e^{-b}\bigr)\Bigr)=
D_\mathrm{KL}\Bigl(\operatorname{Ber}\bigl(e^{-a}\bigr)\Big|\Big|\operatorname{Ber}\bigl(e^{-b}\bigr)\Bigr)
\leq \frac{|a-b|^2}{2 \min\{a,b\}}
\end{equation}
holds for every $a,b>0$.
\end{lemma}

\begin{proof}[Proof of \cref{lem:Bernoulli_DKL}]
We have by definition that
\begin{align*}D_\mathrm{KL}\Bigl(\operatorname{Ber}(p)\Big|\Big|\operatorname{Ber}(q)\Bigr) &=  p \log \frac{p}{q} + (1-p) \log \frac{1-p}{1-q} = \int_q^p \frac{p-s}{s(1-s)} \dif s
\end{align*}
for every $0<p,q<1$. First assume that $a<b$. Taking $p=e^{-a}$ and $q=e^{-b}$ and using the substitution $s=e^{-t}$ yields that
\begin{align*}D_\mathrm{KL}\Bigl(\operatorname{Ber}\bigl(e^{-a}\bigr)\Big|\Big|\operatorname{Ber}\bigl(e^{-b}\bigr)\Bigr)= \int_a^b \frac{e^{-a}-e^{-t}}{1-e^{-t}} \dif t \leq \frac{e^{-a}}{1-e^{-a}}\int_0^{b-a} (1-e^{-t})\dif t.
\end{align*}
Using that $1-e^{-t} \leq t$ and $e^{-t}/(1-e^{-t}) \leq 1/t$ for $t\geq 0$, we deduce that
\begin{align*}D_\mathrm{KL}\Bigl(\operatorname{Ber}\bigl(e^{-a}\bigr)\Big|\Big|\operatorname{Ber}\bigl(e^{-b}\bigr)\Bigr) \leq \frac{e^{-a}(b-a)^2}{2(1-e^{-a})} \leq \frac{|a-b|^2}{2\min\{a,b\}}
\end{align*}
as claimed. Now assume that $b<a$. Taking $p=e^{-a}$ and $q=e^{-b}$ and using the substitution $s=e^{-t}$ yields that
\begin{align*}D_\mathrm{KL}\Bigl(\operatorname{Ber}\bigl(e^{-a}\bigr)\Big|\Big|\operatorname{Ber}\bigl(e^{-b}\bigr)\Bigr)= \int_b^a \frac{e^{-t}-e^{-a}}{1-e^{-t}} \dif t 
\leq \frac{e^{-a}}{1-e^{-b}}\int_0^{a-b} (e^{t}-1)\dif t.
\end{align*}
Now, using the inequality $e^{-x}(e^x-1) \leq x$ we deduce that $e^{-a}(e^t-1) \leq e^{-b} t$ for every $0\leq t \leq a-b$ and hence that
\begin{align*}D_\mathrm{KL}\Bigl(\operatorname{Ber}\bigl(e^{-a}\bigr)\Big|\Big|\operatorname{Ber}\bigl(e^{-b}\bigr)\Bigr) 
\leq \frac{e^{-b}}{1-e^{-b}}\int_0^{a-b} t\dif t \leq \frac{e^{-b}(b-a)^2}{2(1-e^{-b})} \leq \frac{|a-b|^2}{2\min\{a,b\}}
\end{align*}
as before.
\end{proof}

We are now ready to prove \cref{thm:DewanMuirhead}.

\begin{proof}[Proof of \cref{thm:DewanMuirhead}]

For notational convenience, for each $\omega\in \{0,1\}^E$ we will set $T_i(\omega)=\varnothing$ for $i\geq \tau(\omega)$ so that $T_i(\omega)$ is well-defined for every $\omega \in \{0,1\}^E$ and $i\geq 1$.
Let $\omega_1,\omega_2 \in \{0,1\}^E$ be samples of Bernoulli percolation on $G$ with parameters $\beta_1$ and $\beta_2$ respectively, and let $T_1^1=T_1(\omega_1), T_2^1=T_2(\omega_1), \ldots, T_{\tau_1}^1=T_{\tau(\omega_1)}(\omega_1)$ and $T_1^2=T_2(\omega_2), T_2^2=T_2(\omega_2), \ldots, T_{\tau_2}^2=T_{\tau(\omega_2)}(\omega_2)$ be the edges of $G$ that are revealed when applying the decision tree $T$ to $\omega_1$ and $\omega_2$ respectively. 
For each $i\geq 1$, let $X_i=\omega_1(T_i^1)$ if $i \leq \tau_1$ and $X_i = \varnothing$ otherwise, so that $X=(X_1,X_2,\ldots)$ is an infinite $\{0,1,\varnothing\}$-valued sequence, and similarly define $Y_i = \omega_2(T_i^2)$ if if $i \leq \tau_2$ and $Y_i = \varnothing$ otherwise. For each $i\geq 1$, let $X^i=(X_1,\ldots,X_i)$ and $Y^i=(Y_1,\ldots,Y_i)$. For each $k\geq 1$, we have by the chain rule that
\begin{multline}
\label{eq:chainruleproof}
D_\mathrm{KL}\bigl(X^{k+1} || Y^{k+1}\bigr)
=D_\mathrm{KL}\bigl(X^{k} || Y^{k}\bigr)
+ \E_{x \sim X^k} \left[D_\mathrm{KL}\Bigl( \bigl(X_{k+1} \mid X^{k} = x \bigr) \Big|\Big| \bigl(Y_{k+1} \mid Y^k= x \bigr)  \Bigr)\right].
\end{multline}
Since $T$ is a decision tree, for each $x\in \{0,1,\varnothing\}^k$ we either have that $X_{k+1}=\varnothing$ and $Y_{k+1}=\varnothing$ on the events $X^k=x$ and $Y^k=x$ or else that there exists an edge $T_{k+1}(x)$ such that $X_{k+1}=\omega_1(T_{k+1}(x))$ and $Y_{k+1}=\omega_2(T_{k+1}(x))$ on the events $X^k=x$ and $Y^k=x$. Letting $J_{k+1}(x)$ be the weight of the edge $T_{k+1}(x)$, it follows that 
\begin{align*}
&D_\mathrm{KL}\Bigl( \bigl(X_{k+1} \mid X^{k} = x \bigr) \Big|\Big| \bigl(Y_{k+1} \mid Y^k= x \bigr)  \Bigr) 
\\&\hspace{4.5cm}= 
\mathbbm{1}(T_{k+1}(x) \neq \varnothing) D_\mathrm{KL}\Bigl(\operatorname{Ber}\bigl(1-e^{-\beta_1 J_{k+1}(x)}\bigr)\Big|\Big|\operatorname{Ber}\bigl(1-e^{-\beta_2 J_{k+1}(x)}\bigr)\Bigr)
\\&\hspace{4.5cm}
\leq
\mathbbm{1}(T_{k+1}(x) \neq \varnothing) \frac{|\beta_1-\beta_2|^2}{2\min\{\beta_1,\beta_2\}} J_{k+1}(x).
\end{align*}
for every $x\in \{0,1,\varnothing\}^k$ such that $X^k=x$ with positive probability. Taking expectations and letting $J_i$ be the weight of the edge $T_i^1$ for each $1\leq i \leq \tau$, it follows by induction on $k$ that
\begin{equation}
D_\mathrm{KL}(X^k || Y^k) \leq \frac{|\beta_1-\beta_2|}{2\min\{\beta_1,\beta_2\}}\E_{\beta_1} \sum_{i=1}^{k\wedge \tau}J_{i} 
\end{equation}
for every $k\geq 1$. Taking the limit as $k\uparrow \infty$ (which is valid since the relative entropy $D_\mathrm{KL}(\mu ||\nu)$ is lower semicontinuous in $(\mu,\nu)$ with respect to the weak topology on probability measures \cite[Theorem 1]{MR689214}), we deduce that
\begin{equation}
D_\mathrm{KL}(X|| Y) 
\leq \frac{|\beta_1-\beta_2|}{2\min\{\beta_1,\beta_2\}}\E_{\beta_1} \sum_{i=1}^{\tau}J_{i}
= \frac{|\beta_1-\beta_2|^2}{2\min\{\beta_1,\beta_2\}} \sum_{e\in E} J_e \cdot \operatorname{Rev}_{\beta_1}(T,e).
\end{equation}
Now, since $T$ Borel-computes the event $A$, there exists a Borel-measurable subset $\sA$ of $\{0,1,\varnothing\}^\N$ such that $\omega_1$ belongs to $A$ if and only if $X$ belongs to $\sA$
and $\omega_2$ belongs to $A$ if and only if $Y$ belongs to $\sA$. The claim now follows by applying the generalised Pinsker  inequality \eqref{eq:Pinsker} to the laws of $X$ and $Y$ and the event $\sA$.
\end{proof}

\begin{proof}[Proof of \cref{thm:DM_Chi}]
We will prove the stronger inequality 
\begin{equation}
\label{eq:DM_Chi_strong2}
\P_{\beta_2}(|K|\geq n) \leq 2\P_{\beta_1}(|K|\geq n) + \frac{4}{\beta_1}|\beta_2-\beta_1|^2 \sum_{k=1}^n \P_{\beta_1}(|K|\geq k).
\end{equation}
as stated in \eqref{eq:DM_Chi_strong}, from which the claimed inequality follows by Markov's inequality.
Fix $0< \beta_1 < \beta_2$ and $n\geq 0$. 
We apply \cref{thm:DewanMuirhead} taking $A=\{|K|\geq n\}$ and taking $T$ to be a decision tree that explores the cluster of the origin one edge at a time, stopping if and when it first finds $n$ vertices in $K$; see e.g.\ the proof of \cite[Proposition 3.1]{1901.10363} for a formal definition. This decision tree can only query edges with at least one endpoint in $K$, and the edges queried by $T$ all have endpoints in some set of size at most $n$. Using the assumption that $\sum_{e\in E^\rightarrow_o} J_e =1$, it follows that
\[
\sum_{e\in E} J_e \cdot \operatorname{Rev}_{\beta_1}(T,e) \leq \E_{\beta_1} \left[ |K| \wedge n \right] =\sum_{k=1}^n \P_{\beta_1}(|K|\geq k)
\]
and hence by \cref{thm:DewanMuirhead} that
\begin{equation}
|\P_{\beta_2}(|K|\geq n)-\P_{\beta_1}(|K|\geq n)|^2 \leq \frac{1}{ \beta_1 } |\beta_2-\beta_1|^2  \P_{\beta_2}(|K|\geq n) \sum_{k=1}^n \P_{\beta_1}(|K|\geq k).
\label{eq:DM_Chi_almost_done}
\end{equation}
If $\P_{\beta_2}(|K|\geq n) \leq 2\P_{\beta_1}(|K|\geq n)$ then \eqref{eq:DM_Chi_strong2} holds trivially, while if $\P_{\beta_2}(|K|\geq n) \geq 2\P_{\beta_1}(|K|\geq n)$ then it follows from \eqref{eq:DM_Chi_almost_done} that
\[
 \P_{\beta_2}(|K|\geq n)^2 \leq 4 |\P_{\beta_2}(|K|\geq n)-\P_{\beta_1}(|K|\geq n)|^2 \leq \frac{4}{ \beta_1 } |\beta_2-\beta_1|^2 \P_{\beta_2}(|K|\geq n) \sum_{k=1}^n \P_{\beta_1}(|K|\geq k),
\]
so that the claimed inequality \eqref{eq:DM_Chi_strong2} holds in this case also. This completes the proof.
\end{proof}

\section{Applications to the hierarchical lattice}
\label{sec:hierarchical_proof}
In this section we prove our results concerning long-range percolation on the hierarchical lattice. In addition to \cref{thm:d/3}, we will also prove the following related theorem showing that the critical exponents describing the model are close to their mean-field values when $\alpha$ is only slightly larger than $d/3$. Note that when $\alpha>d/3$ the model is below its upper-critical dimension and mean-field critical behaviour does not hold \cite[Corollary 1.5]{hutchcroft2021critical}.

\begin{thm}
\label{thm:>d/3}
Let $J:\mathbb{H}^d_L \to [0,\infty)$ be a radially symmetric, integrable function, let $0<\alpha<d$, and suppose that there exist constants $c$ and $C$ such that $c \|x\|^{-d-\alpha} \leq J(x) \leq C\|x\|^{-d-\alpha}$ for every $x\in \mathbb{H}^d_L \setminus \{0\}$. If $d/3<\alpha<4d/11$ then
\begingroup
\begin{align}
\Chi_{\beta_c-\eps} &\preceq \eps^{-\alpha/(2d-5\alpha)} &\text{ as }\eps &\downarrow 0,
\\
\P_{\beta_c}\left( |K| \geq n\right) &\preceq n^{-(4d-11\alpha)/(4d-10\alpha)}&\text{ as }n&\uparrow \infty,  \text{ and}
\\
\P_{\beta_c+\eps}\left(|K|=\infty\right) &\preceq  \eps^{(4d-11\alpha)/(2d-5\alpha)}&\text{ as }\eps &\downarrow 0.
\end{align}
\endgroup
\end{thm}

The requirement that $\alpha < 4d /11$ is used only to ensure that the resulting bounds on $\P_{\beta_c}(|K|\geq n)$ and $\P_{\beta_c+\eps}(|K|=\infty)$ are non-trivial; we expect this threshold at $4d/11$ to be a feature of the proof that does not correspond to any true change in behaviour of the model. 
%
%
Indeed, power-law upper bounds on the same quantities have been proven for all $0<\alpha<d$ via completely different methods in our recent work \cite{hutchcroft2020power}. (See in particular \cite[Theorem 1.2]{hutchcroft2020power} for a very general result applying to the hierarchical lattice.) That paper does \emph{not} establish that these exponents approach their mean-field values as $\alpha \downarrow d/3$, or indeed that mean-field behaviour holds for $\alpha <d/3$. See \cref{fig:1d} for a detailed comparison of the two sets of results.

\pgfplotsset{width=0.45\textwidth}
\pgfplotsset{compat=newest}

\begin{figure}[t]
\centering
    \begin{tikzpicture}
\begin{axis}[
    axis lines = left,
    xlabel = $\alpha/d$,
    ylabel = {$\delta$},
    xmin=0, xmax=1,
    ymin=0, ymax=25
]
\addplot [
    domain=0:1/3, 
    samples=100, 
    color=purple,
    semithick
]
{2};
\addplot [
    domain=1/3:0.99, 
    samples=100, 
    color=red,
    semithick
]
{(1+x)/(1-x)};
\addplot [
    domain=0:0.99, 
    samples=100, 
    color=orange,
    semithick
    ]
    {(2+x)/(1-x)};
    \addplot [
    domain=0:0.99, 
    samples=100, 
    color=teal,
    semithick
    ]
    {2/(1-x)};
        \addplot [
    domain=1/3:0.363, 
    samples=100, 
    color=blue,
    semithick
    ]
    {(4-10*x)/(4-11*x)};
\end{axis}
\end{tikzpicture}
\hspace{1cm}
    \begin{tikzpicture}
\begin{axis}[
    axis lines = left,
    xlabel = $\alpha/d$,
    ylabel = {$\gamma$},
    xmin=0, xmax=1,
    ymin=0, ymax=25
]
\addplot [
    domain=0:1/3, 
    samples=100, 
    color=purple,
    semithick
]
{1};
\addplot [
    domain=0:0.99, 
    samples=100, 
    color=orange,
    semithick
    ]
    {(2+x)/(1-x)-1};
    \addplot [
    domain=0:0.99, 
    samples=100, 
    color=teal,
    semithick
    ]
    {2/(1-x)-1};
        \addplot [
    domain=1/3:0.399, 
    samples=100, 
    color=blue,
    semithick
    ]
    {x/(2-5*x)};
\end{axis}
\end{tikzpicture}
    \caption{
Estimates on the exponents (left) $\delta$ and $\gamma$ (right) defined by $\P_{\beta_c}(|K|\geq n) \approx n^{-1/\delta}$ as $n\to\infty$ and $\Chi_{\beta_c-\eps}\approx \eps^{-\gamma}$ as $\eps \downarrow 0$. Purple: The mean-field regime $d>3\alpha$, where $\delta=2$ and $\gamma=1$. Red: The conjectured true value $\delta=(d+\alpha)/(d-\alpha)$ for $d \leq 3\alpha$ (we are not aware of a conjectured exact value for $\gamma$ outside of the mean-field regime). Blue: The upper bounds $\delta \leq (4d-10\alpha)/(4d-11\alpha)$ and $\gamma \leq \alpha/(2d-5\alpha)$ of \cref{thm:d/3,thm:>d/3}. Orange: The upper bound $\delta \leq (2d+\alpha)/(d-\alpha)$ of \cite{hutchcroft2020power} and the associated bound on $\gamma$ obtained using the inequality $\gamma \leq \delta-1$ of \cite{1901.10363}. Teal: The upper bound $\delta \leq 2d/(d-\alpha)$ obtainable by inputting the two-point estimate of \cite{hutchcroft2021critical} into the method of \cite{hutchcroft2020power} as done in \cite[Section 3]{hutchcroft_sharp_LRP} and the associated bound on $\gamma$ obtained using the inequality $\gamma \leq \delta-1$.
}
    \label{fig:1d}
    \vspace{-1em}
\end{figure}

\medskip

For the remainder of the section we will fix $d\geq 1$, $L \geq 2$, $0<\alpha <d$ and a radially symmetric function $J:\bbH^d_L \to [0,\infty)$ satisfying $c \langle x \rangle^{-d-\alpha} \leq J(x) \leq C \langle x \rangle^{-d-\alpha}$ for some positive constants $c$ and $C$ and every $x\in \bbH^d_L \setminus \{0\}$.
We will assume without loss of generality that $\sum_{x\in \bbH^d_L} J(x)=1$ and write $\asymp$, $\preceq$, and $\succeq$ for equalities and inequalities holding to within positive multiplicative constants depending only on $d,L,\alpha,c$ and $C$.

\medskip

 Recall that for each $n \geq 1$ we write $\Lambda_n=\{x\in \bbH^d : \langle x \rangle \leq L^n\}$ for the ultrametric ball of radius $L^n$ containing the origin, noting that the restriction to $\Lambda_n$ of the weighted graph defined in terms of $\bbH^d_L$ and $J$ is itself transitive. (Indeed, the ball $\Lambda_n$ has the group structure of the torus $\bigoplus_{i=1}^n \mathbb{T}^d_L$.) Let $K_n$ denote the cluster of the origin in $\Lambda_n$, i.e., the set of vertices in $\Lambda_n$ that are connected to the origin by an open path contained in $\Lambda_n$. 
For each $\beta \geq 0$ and $n\geq 0$ we let $T_{\beta,n}$ denote the two-point matrix on $\Lambda_n$ and define $\Chi_{\beta,n} := \E_\beta |K_n| = \sum_{x\in \Lambda_n} T_{\beta,n}(0,x)$ and $\nabla_{\beta,n}:= \sum_{x,y \in \Lambda_n} T_{\beta,n}(0,x)T_{\beta,n}(x,y)T_{\beta,n}(y,0)$. It is proven in \cite[Corollary 1.2]{hutchcroft2021critical} that
\begin{equation}
\label{eq:hierarchical_Chi}
\Chi_{\beta_c,n} \asymp L^{\alpha n}
\end{equation}
for every $n \geq 0$, while 
 \cref{thm:main_chi} yields that
\begin{equation}
\label{eq:hierarchical_diffineq}
\frac{d \Chi_{\beta,n}}{d\beta}   \geq \frac{\Chi_{\beta,n}(\Chi_{\beta,n}-\nabla_{\beta,n})}{3 \beta^2 \nabla_{\beta,n}^2} 
\end{equation}
for every $\beta \geq 0$ and $n\geq 0$.

\medskip

We will apply the differential inequality \eqref{eq:hierarchical_diffineq} to study the growth of the \emph{correlation length} $\xi(\beta)$ as $\beta \uparrow \beta_c$. Before defining the correlation length, we first develop some general theory that will motivate the definition. Let $G=(V,E,J)$ be a countable weighted graph.
Following Duminil-Copin and Tassion \cite{duminil2015new}, for each $v\in V$, $\beta\geq 0$, and finite subset $S \subseteq V$  we consider the quantity
\[
\phi_\beta(S,v) := \sum_{e \in \partial^\rightarrow S} \Bigl(1-e^{-\beta J_e}\Bigr)\P_\beta(v \xleftrightarrow{S} e^-),
\]
where we write $\{x \xleftrightarrow{S} y\}$ to mean that $x$ and $y$ are connected by an open path all of whose vertices belong to $S$ and recall that $\partial^\rightarrow S$ denotes the set of oriented edges with $e^-\in S$ and $e^+\notin S$. When $G$ is infinite and transitive, 
Duminil-Copin and Tassion  \cite{duminil2015new} proved that the critical parameter $\beta_c$ admits the alternative characterisation
\begin{equation}
\label{eq:phibeta_betac}
\beta_c = \inf\Bigl\{ \beta \geq 0 : \phi_\beta(S,o) \geq 1 \text{ for every finite $S\ni o$}\Bigr\},
\end{equation}
where $o$ is an arbitrary fixed root vertex.
This characterisation is closely related to the following inequality, which slightly sharpens the analysis of \cite[Theorem 1.1, Item 2]{duminil2015new}.

\begin{lemma}
\label{lem:phi_beta_S}
Let $G=(V,E,J)$ be a countable weighted graph and let $S\subseteq \Lambda$ be finite subsets of $V$. Then
\[
\sum_{x\in \Lambda} \P_\beta(v \xleftrightarrow{\Lambda} x) \leq \sum_{x\in S} \P_\beta(v \xleftrightarrow{S} x) + \phi_\beta(S,v) \cdot \sup_{u\in \Lambda} \sum_{x\in \Lambda} \P_\beta(u \xleftrightarrow{\Lambda} x)
\]
for every $\beta\geq 0$ and $v\in S$.
\end{lemma}

\begin{proof}[Proof of \cref{lem:phi_beta_S}]
 Fix $\beta \geq 0$ and $v\in S$. For each $x\in \Lambda$, considering the edge crossed by an open path from $v$ to $x$ in $\Lambda$ as it leaves $S$ for the first time yields that
\[
\{v\xleftrightarrow{\Lambda} x\} \setminus \{v \xleftrightarrow{S} x\} \subseteq \bigcup_{e\in \partial^\rightarrow S} \{v \xleftrightarrow{S} e^-\}\circ\{e \text{ open}\} \circ \{e^+ \xleftrightarrow{\Lambda} x\}.
\] 
Thus, it follows by the BK inequality and a union bound that
\[
\P_\beta(v \xleftrightarrow{\Lambda} x) \leq \P(v \xleftrightarrow{S} x) + \sum_{e\in \partial^\rightarrow S} \P_\beta(v \xleftrightarrow{S} e^-) (1-e^{-\beta J_e})\P_\beta(e^+ \xleftrightarrow{\Lambda} x)
\]
for every $x\in \Lambda$. The claim follows by summing over $x \in \Lambda$.
\end{proof}

We now return to our setting of the hierarchical lattice as above. For each $n \geq 0$ we define
\[
\beta_n = \sup\left\{\beta \geq 0 : \phi_\beta(\Lambda_n,0) \leq \frac{1}{2} \right\} \qquad \text{ and } \qquad \beta^*_n = \max_{0 \leq m \leq n} \beta_m,
\]
so that $(\beta_n^*)_{n\geq 0}$ is a non-decreasing sequence.
Since $\phi_{\beta_c}(\Lambda_n,0) \geq 1$ for every $n\geq 0$ by \eqref{eq:phibeta_betac}, we have that $\beta^*_n< \beta_c$ for every $n\geq 0$.
For each $0\leq \beta<\beta_c$ we define the \textbf{correlation length} $\xi(\beta)$ by
\begin{equation}
\label{eq:correlation_length_definition}
\xi(\beta) = L^{n(\beta)} \qquad \text{ where } \qquad  n(\beta)=\inf\{n\geq 0 : \beta \leq \beta_n^*\}.
\end{equation}
We are justified in referring to this quantity as the correlation length by the following lemma, which also implies that $\xi(\beta) \uparrow \infty$ as $\beta \uparrow \beta_c$.

\begin{lemma}
\label{lem:thewindow}
$\Chi_\beta \asymp \Chi_{\beta,n} \asymp \Chi_{\beta_c,n} \asymp L^{\alpha n}$ for every $n\geq 0$ and $\beta_n^* \leq \beta \leq \beta_{n+1}^*$.
\end{lemma}

 Roughly speaking, this lemma shows that the two-point function is comparable to the critical two-point function for $x$ with $\langle x \rangle \leq \xi(\beta)$ and that scales above the correlation length do not contribute significantly to the susceptibility.

\begin{proof}[Proof of \cref{lem:thewindow}]
It follows from \cref{lem:phi_beta_S} and transitivity of $\Lambda_n$ that 
\[
\Chi_{\beta,N} \leq \Chi_{\beta,n} + \phi_\beta(\Lambda_n,0) \cdot \Chi_{\beta,N}
\]
for each $N \geq n \geq 0$ and $\beta \geq 0$, and hence that
\begin{equation}
\label{eq:Chi_in_the_window}
\Chi_\beta = \lim_{N\to\infty} \Chi_{\beta,N} \leq 2 \Chi_{\beta,n} \leq 2 \Chi_{\beta_c,n}
\end{equation}
for every $n \geq 0$ and $0\leq \beta \leq \beta_n$. In particular, if $\beta \leq \beta_{n+1}^*$ then $\beta\leq \beta_m$ for some $0\leq m \leq n+1$ and we have by \eqref{eq:hierarchical_Chi} and \eqref{eq:Chi_in_the_window} that
\[
\Chi_\beta \leq 2 \Chi_{\beta_c,m} \leq 2 \Chi_{\beta_c,n+1} \preceq L^{\alpha(n+1)} \preceq L^{\alpha n}.
\]
This implies that $ \Chi_{\beta,n} \leq \Chi_\beta \preceq \Chi_{\beta_c,n} \asymp L^{\alpha n}$ for every $\beta \leq \beta_{n+1}^*$. To complete the proof it suffices to prove that if $\beta \geq \beta_n^*$ then $\Chi_{\beta,n}\succeq L^{\alpha n}$. To this end, note that if $\beta\geq \beta_n^*$ then $\beta \geq \beta_n$ and hence by the definitions and \eqref{eq:hierarchical_Chi} that
\[
\frac{1}{2} \leq \phi_\beta(\Lambda_n,0) =  \sum_{y\in \bbH^d_L \setminus \Lambda_n} \Bigl(1-e^{-\beta J(y)} \Bigr) \cdot \E_\beta |K_n| \asymp \beta L^{-\alpha n} \cdot \E_\beta |K_n| = \beta L^{-\alpha n} \Chi_{\beta,n}.
\]
This rearranges to give the claimed lower bound.
\end{proof}


\cref{thm:d/3,thm:>d/3} will be deduced from \cref{lem:thewindow} together with the following proposition.

\begin{prop}
\label{prop:correlation_length}
Let $J:\mathbb{H}^d_L \to [0,\infty)$ be a radially symmetric, integrable function, let $0<\alpha<d$, and suppose that there exist constants $c$ and $C$ such that $c \|x\|^{-d-\alpha} \leq J(x) \leq C\|x\|^{-d-\alpha}$ for every $x\in \mathbb{H}^d_L \setminus \{0\}$. If $\alpha<2d/5$ then
\[
\xi(\beta) \preceq \begin{cases}
|\beta-\beta_c|^{-1/\alpha} & \alpha < d/3\\
|\beta-\beta_c|^{-1/\alpha}\left(\log \frac{1}{|\beta-\beta_c|}\right)^{2/\alpha}  & \alpha = d/3\\
|\beta-\beta_c|^{-1/(2d-5\alpha)} & \alpha > d/3
\end{cases}
\]
for every $0 \leq \beta < \beta_c$.
\end{prop}

\begin{remark}
It follows from \cref{lem:thewindow} together with the mean-field lower bound $\Chi_\beta \succeq |\beta-\beta_c|^{-1}$ that the correlation length always satisfies the mean-field lower bound $\xi(\beta) \succeq (\beta_c-\beta)^{-1/\alpha}$ for every $0\leq \beta < \beta_c$, no matter what value $\alpha$ takes between $0$ and $d$.
\end{remark}

\begin{proof}[Proof of \cref{prop:correlation_length}]
It suffices to prove that there exists a constant $n_0$ such that the estimate
\[
|\beta_n^* - \beta_c| \preceq \begin{cases} L^{-\alpha n} & \alpha < d/3\\
(n+1)^2 L^{-\alpha n} & \alpha = d/3\\
L^{-(2d-5\alpha)n} & \alpha > d/3
\end{cases}
\]
holds for every $n\geq n_0$.
It follows from \cref{lem:thewindow} and \eqref{eq:hierarchical_Chi} that there exist positive constants $c_1$ and $C_1$ such that $c_1 L^{\alpha n} \leq \Chi_{\beta,n} \leq C_1L^{\alpha n}$ for every $n\geq 0$ and $\beta_n^* \leq \beta \leq \beta_c$.
On the other hand, it is an immediate consequence of \eqref{eq:hierarchical_Chi} as shown in \cite[Equation 2.19]{hutchcroft2021critical} that
 \begin{equation}
 \label{eq:nabla_n}
1\leq \nabla_{\beta,n} \leq \sum_{x,y \in \Lambda_n} T_{\beta_c}(0,x)T_{\beta_c}(x,y)T_{\beta_c}(y,0) \preceq \begin{cases}
1 & \alpha < d/3\\
n+1 & \alpha = d/3\\
L^{(3\alpha-d)n} & \alpha > d/3
\end{cases}
 \end{equation}
for every $n\geq 0$ and $0 \leq \beta \leq \beta_c$. Since $\alpha < 2d/5 \leq d/2$ we deduce that $\nabla_{\beta_c,n} = o(\Chi_{\beta_n^*,n})$ as $n\to \infty$ and hence by \eqref{eq:hierarchical_diffineq} that there exist positive constants $c_2$, $c_3$ and $n_0$ such that
\[
\frac{ d \Chi_{\beta,n}}{d \beta} \geq c_2 \frac{\Chi_{\beta,n}^2}{\nabla_\beta^2} \geq c_3 \cdot \begin{cases}
L^{2\alpha n} & \alpha < d/3\\
\frac{1}{(n+1)^2} L^{2\alpha n} & \alpha = d/3\\
L^{(2d-4\alpha)n} & \alpha > d/3
\end{cases}
\]
for every $n\geq n_0$ and every $\beta_n^* \leq \beta \leq \beta_c$. Integrating this inequality between $\beta_n^*$ and $\beta_c$ yields that
\[
\Chi_{\beta_c,n} \geq c_3 |\beta_c-\beta_n^*| \cdot \begin{cases}
L^{2\alpha n} & \alpha < d/3\\
\frac{1}{(n+1)^2} L^{2\alpha n} & \alpha = d/3\\
L^{(2d-4\alpha)n} & \alpha > d/3
\end{cases}
\]
and the claim follows by rearranging and applying \eqref{eq:hierarchical_Chi} to estimate the left hand side.
\end{proof}

\begin{proof}[Proof of \cref{thm:d/3,thm:>d/3}]
It follows immediately from \cref{lem:thewindow} and \eqref{prop:correlation_length} that
\[
\Chi_\beta \preceq \xi(\beta)^\alpha \preceq 
 \begin{cases}
|\beta-\beta_c|^{-1} & \alpha < d/3\\
|\beta-\beta_c|^{-1}\left(\log \frac{1}{|\beta-\beta_c|}\right)^{2}  & \alpha = d/3\\
|\beta-\beta_c|^{-\alpha/(2d-5\alpha)} & \alpha > d/3
\end{cases}
\]
for every $0<\beta < \beta_c$. The remaining claims follow easily from \cref{thm:DM_Chi} as in the proof of \cref{cor:logs}.
\end{proof}

\begin{remark}
The tree-graph estimates \eqref{eq:threepoint} and \eqref{eq:fourpoint} on the three- and four-point functions used to prove \eqref{eq:diagram_derivation1} are not sharp outside of the mean-field regime, and one can  improve the exponent estimates of \cref{thm:>d/3} by instead using the techniques of \cite{hutchcroft2021critical} to directly bound the critical four-point function in the hierarchical case. The improvements we were able to obtain via this method were rather modest and we have chosen not to pursue this further here.
\end{remark}

\subsection*{Acknowledgements}
This work was carried out while the author was a Senior Research Associate at the University of Cambridge, and was supported in part by ERC starting grant 804166 (SPRS). We thank Vivek Dewan, Emmanuel Michta, Stephen Muirhead, and Gordon Slade for helpful comments on a previous version of the manuscript.

  \bibliographystyle{abbrv}
  \bibliography{unimodularthesis.bib}

\appendix

\section{Comparing diagrams: a sojourn into linear algebra}
\label{sec:linear_algebra}

In this section we prove \cref{prop:triangle_comparison} in full generality, removing all the Fourier analysis from the proof given for the Euclidean and hierarchical cases in \cref{sec:mainproof}. 
We will use the alternative expressions for $A_\beta$ and $B_\beta$ derived from the mass-transport principle in \eqref{eq:Adefunimod} and \eqref{eq:Bdefunimod} throughout our calculations.
While we have made this section into an appendix due to its tangential relationship to the rest of the paper, we stress that the tools introduced here are in our view highly elegant and, in some cases, do not appear to have been used in probability theory or classical statistical mechanics before. 

\medskip

The general proof of \cref{prop:triangle_comparison} will be based on the theory of infinite positive (semi)definite matrices and the Schur product theorem, with a particularly important role being played by a theorem of Ando \cite{MR535686} concerning the interplay between the \emph{Hadamard product} (i.e., the entrywise product) and the usual matrix product. Since the results we use are spread around several disparate sources and since we expect the content of this section to be highly unfamiliar to most of our audience, we give a thorough and self-contained treatment of most of the results that we use. A secondary purpose of this self-contained account is to prove infinite-dimensional cases of various results that are usually stated only in the finite-dimensional case. In addition to the paper on Ando \cite{MR535686}, the primary sources we used to prepare this section were \cite{MR2284176} and \cite[Appendix A]{MR2064921}, both of which are written in a very accessible manner.

\medskip

Let $V$ be a countable (finite or infinite) set, and denote by $\cM(V) = \R^{V \times V}$ the set of all matrices indexed by $V$. Let $L^0(V)$ be the set of \emph{finitely supported} functions $L^0(V) := \{ f \in \R^V : f(v) =0 $ for all but finitely many $v\}$, so that every $T \in \cM(V)$ defines a linear map
\[\begin{array}{rcl} L^0(V) &\longrightarrow & \R^V\\
f & \longmapsto & Tf
\end{array} \qquad  \text{ where } \qquad Tf(u)=\sum_{v\in V} T(u,v) f(v) \quad \text{ for each $u\in V$}
\]
in the natural way, where the assumption that $f \in L^0(V)$ means that convergence is not an issue.

\medskip

A matrix $T \in \cM(V)$ is said to be \textbf{symmetric} if $T(u,v) = T(v,u)$ for every $u,v \in V$. A matrix $T \in \cM(V)$ is said to be \textbf{positive semidefinite} if it is symmetric and satisfies $\langle T f, f \rangle := \sum_{u,v \in V} f(u) T(u,v) f(v) \geq 0$ for every $f \in L^0(V)$  and is said to be \textbf{positive definite} if there exists $\lambda>0$ such that $\langle T f, f \rangle \geq \lambda \langle f,f \rangle$ for every $f\in L^0(V)$. 
(Again, restricting to $f \in L^0(V)$ means that both inner products are trivially well-defined.) Given $S,T \in \cM(V)$, we write $S \geq T$ to mean that $S-T$ is positive semidefinite and write $S>T$ to mean that $S-T$ is positive definite.


\medskip

The relevance of these notions to percolation theory is established by the following lemma of Aizenman and Newman \cite[Lemma 3.3]{MR762034}. We include a proof for completeness.

\begin{lemma}[Aizenman and Newman]
\label{lem:positive_definite}
Let $G$ be a countable weighted graph. Then $T_\beta$ is positive \emph{semidefinite} for every $\beta \geq 0$. Moreover, if $G$ has $\sup_{v\in V} \sum_{e\in E^\rightarrow_v} J_e < \infty$ then $T_\beta$ is positive \emph{definite} for every $\beta \geq 0$.
\end{lemma}

\begin{proof}[Proof of \cref{lem:positive_definite}]
The matrix $T_\beta$ is trivially symmetric. Let $\sC$ be the set of clusters of $\omega$, each of which is a set of vertices. Then for each finitely supported $f: V \to \R$ we have that
\begin{align}
\langle T_\beta f, f \rangle &= \sum_{u,v \in V} f(u)f(v)\P_\beta(u \leftrightarrow v) = \E_\beta\left[ \sum_{u,v \in V} f(u)f(v)\mathbbm{1}(u \leftrightarrow v) \right]
\nonumber
\\
&= \E_\beta\left[ \sum_{C \in \sC} \Bigl(\sum_{u\in C} f(u) \Bigr)^2\right] \geq 0,
\label{eq:positive_definite_proof}
\end{align}
so that $T_\beta$ is positive semidefinite as required. Similarly, letting 
 $\sC_1 \subseteq \sC$ be the set of singleton clusters of $\omega$, that is, clusters containing exactly one vertex, we have by \eqref{eq:positive_definite_proof} that
\begin{multline}
\langle T_\beta f, f \rangle \geq \E_\beta\left[ \sum_{C \in \sC_1} \Bigl(\sum_{u\in C} f(u) \Bigr)^2\right] = \sum_{u \in V} f(u)^2 \P_\beta(\{u\} \text{ is a singleton cluster}\}) \\ =  \sum_{u \in V} f(u)^2 \prod_{e\in E^\rightarrow_v} e^{-\beta J_e}
\end{multline}
which is  easily seen to establish the second claim.
\end{proof}

\begin{remark}
The proof that the two-point matrix is positive semidefinite works for \emph{any} percolation process, that is, any random subgraph of a given graph, and is not specific to Bernoulli percolation. Similarly, the proof of positive definiteness applies to any percolation process in which the probability that a vertex lies in a singleton cluster is bounded away from zero.
\end{remark}


Since we will want to consider \emph{powers} of positive semidefinite matrices, it will be useful to work within an algebra where such powers are well-defined.
Define $\cM^2(V) \subseteq \cM(V)$ to be the algebra of $V$-indexed matrices $T$ such that 
%
%
\[
\|T\|_{2\to 2} := \sup \left\{ \frac{\langle Tf,Tf\rangle^{1/2}}{\langle f,f\rangle^{1/2}} : f\in L^0(V), \langle f ,\, f \rangle > 0 \right\} < \infty.
\]
Since $L^0(V)$ is dense in $L^2(V)$, each matrix $T\in \cM^2(V)$ defines a unique bounded linear operator $f\mapsto Tf$ on $L^2(V)$. Conversely, we can always recover a matrix from its associated operator using the formula $T(u,v) = \langle T \mathbbm{1}_v,\mathbbm{1}_u\rangle$; for $S,T \in \cM^2$ the product of $S$ and $T$ as operators coincides with the usual matrix product $ST(u,v)=\sum_w S(u,w) T(w,v)$.
A symmetric matrix $T \in \cM^2(V)$ is positive semidefinite if and only if $\langle T f, f \rangle \geq 0$ for every $f\in L^2(V)$ and is positive definite if and only if there exists $\lambda >0$ such that $\langle Tf, f \rangle \geq \lambda \langle f,f\rangle$ for every $f\in L^2(V)$. 

\medskip

It is a consequence of the Riesz-Thorin theorem (see e.g.\ \cite[Theorem 9.3.3]{MR2343341}) that
\[
\|T\|_{2\to 2} \leq \|T\|_{1\to 1}^{1/2} \|T\|_{\infty \to \infty}^{1/2} := \sqrt{\sup_{u \in V}\sum_{v\in V} |T(u,v)|}\sqrt{\sup_{u \in V}\sum_{v\in V} |T(v,u)|}
\]
for every $T \in \cM(V)$ (see e.g.\ \cite[Theorem 19 and Exercise 20]{van2014notes} for the fact that the $1\to 1$ and $\infty\to\infty$ norms can be expressed as the maximum $\ell^1$ norm of a column and of a row respectively) and when $T$ is symmetric this simplifies to
\[
\|T\|_{2\to 2} \leq \|T\|_{1\to 1} := \sup_{u \in V}\sum_{v\in V} |T(u,v)|.
\]
In particular, it follows that if $G$ is a transitive weighted graph then the associated two-point matrix $T_\beta$ satisfies $\|T_\beta\|_{2\to 2} \leq \|T_\beta\|_{1\to 1} = \Chi_\beta$ for every $\beta\geq 0$ and hence by the sharpness of the phase transition that $T_\beta$ defines a bounded operator on $L^2(V)$ for every $0\leq \beta <\beta_c$. (In fact it is possible in some contexts for this boundedness to continue hold for $\beta$ strictly larger than $\beta_c$, see \cite{MR4162843} for a thorough discussion.)
Moreover, it follows by spectral considerations that if $T \in \cM^2(V)$ is positive (semi)definite then so is $T^k$ for every $k\geq 1$, so that if $G$ is an infinite, connected, locally finite, quasi-transitive graph then $T_\beta^k$ is positive definite for every $0 \leq \beta < \beta_c$ and $k\geq 1$.

\medskip

Let us now prove the easy part of \cref{prop:triangle_comparison}.

\begin{lemma}
\label{lem:Abound}
Let $G$ be a connected, unimodular transitive weighted graph. Then
$A_\beta \leq \nabla_\beta^{2}$ for every $0 \leq \beta \leq \beta_c$.
\end{lemma}

\begin{proof}[Proof of \cref{lem:Abound}]

By left continuity of $T_\beta$, it suffices to prove the claim for every $0 \leq \beta < \beta_c$, in which case $T_\beta$ is both positive semidefinite and belongs to $\cM^2(V)$, so that $T_\beta^3$ is well-defined as an element of $\cM^2(V)$ and positive semidefinite also. Fix one such $\beta$. Since $G$ is unimodular, we have by \eqref{eq:Adefunimod} that 
$A_\beta = \sum_{v\in V} T_\beta(o,v) T_\beta^2(o,v) T_\beta^3(o,v).$ Since $T_\beta^3$ is positive definite, it has the property that 
\[\langle T_\beta^3(\mathbbm{1}_u-\mathbbm{1}_v) , (\mathbbm{1}_u-\mathbbm{1}_v) \rangle = T_\beta^3(u,u) + T_\beta^3(v,v) - 2 T_\beta^3(u,v) \geq 0\] for every $u,v \in V$ and hence by transitivity that
\[
T_\beta^3(u,v) \leq \frac{1}{2}\left[T_\beta^3(u,u) + T_\beta^3(v,v)  \right] = T_\beta^3(o,o)=\nabla_\beta
\]
for every $u,v \in V$. It follows that
\begin{align*}
 A_\beta &\leq  \nabla_\beta \sum_{v\in V} T_\beta(o,v) T_\beta^2(v,o) = \nabla_\beta T_\beta^3(o,o)=\nabla_\beta^2
\end{align*}
as claimed.
\end{proof}


It remains to prove that $B_\beta\leq A_\beta$. This proof will rely on various more sophisticated pieces of technology, which we now begin to introduce.
 Recall that if $S,T \in \cM(V)$  are  $V$-indexed matrices, the \textbf{Hadamard product} $S \circ T \in \cM(V)$ is defined by entrywise multiplication 
\[
[S \circ T](u,v) : = S(u,v)T(u,v) \qquad \text{ for every $u,v \in V$.}
\]
Note that if $S,T \in \cM^2(V)$ then $S\circ T \in \cM^2(V)$ also and indeed satisfies $\|S\circ T\|_{2\to2}\leq \|S\|_{2\to2}\|T\|_{2\to2}$ \cite[Equation 2.2]{MR2284176}.
The Hadamard product may be used to give various alternative expressions for the diagrams $A_\beta$ and $B_\beta$. The most obvious way to do this, which is analogous to what we did in \eqref{eq:FourierA} and \eqref{eq:FourierB}, is to use \eqref{eq:Adefunimod} and \eqref{eq:Bdefunimod} to write
\begin{align*}
A_\beta =\sum_{v\in V} \Bigl[T_\beta \circ T_\beta^2 \circ T_\beta^3\Bigr](o,v) 
\qquad and \qquad 
  B_\beta = \sum_{v\in V} \Bigl[T_\beta^{2} \circ T_\beta^2 \circ T_\beta^2 \Bigr](o,v).
\end{align*}
However, we will see that the equivalent expressions
\begin{align*}
A_\beta =  \Bigl[T_\beta \circ T_\beta^3\Bigr]T_\beta^2(o,o) 
\qquad and \qquad 
  B_\beta = \Bigl[T_\beta^{2} \circ T_\beta^2\Bigr]T_\beta^2(o,o)
\end{align*}
are better suited to our needs. 
We stress that these expressions contain both Hadamard products and ordinary matrix products: for example, $T_\beta^2$ denotes the ordinary square of $T_\beta$ while $[T_\beta \circ T_\beta^3]T_\beta^2$ denotes the ordinary matrix product of $T_\beta \circ T_\beta^3$ with $T_\beta^2$. 

\medskip

The key estimate needed to complete the proof of \cref{prop:triangle_comparison} is the following proposition, which can be thought of as a generalisation of the inequality \eqref{eq:FourierAndo}.

\begin{prop}
\label{prop:AndoTp}
Let $V$ be a countable set and let $T \in \cM^2(V)$ be positive semidefinite. Then
 \[T^2 \circ T^2  \leq T \circ T^3.\]
\end{prop}

We will deduce this proposition as a special case of a theorem of Ando \cite{MR535686}, which in turn is a (non-obvious) consequence of the Schur product theorem. While all these theorems are usually stated for finite matrices, they are also valid for infinite matrices; indeed, we will see that the infinite cases of these theorems can be immediately reduced to the finite cases.

\begin{theorem}[Schur product theorem]
\label{thm:Shurproduct}
Let $V$ be countable, and let $S,T \in \cM(V)$. If $S$ and $T$ are both positive semidefinite, then $S \circ T$ is positive semidefinite.
\end{theorem}


\begin{proof}[Proof of \cref{thm:Shurproduct}]
Recall that if $V$ is a countable set and $T \in \cM(V)$ the matrices $T|_W \in \cM(W)$ where $W \subseteq V$ and $T|_W(u,v)= T(u,v)$ for every $u,v \in W$ are referred to as the \textbf{principal submatrices} of $T$.
By the definitions, a matrix $T \in \cM(V)$ is positive definite if and only if all its finite dimensional principal submatrices are. Thus, since the Hadamard product commutes with taking principal submatrices in the sense that $S|_W \circ T|_W = (S\circ T)|_W$ for every $S,T \in \cM(V)$ and $W \subseteq V$, it suffices to consider the case that $V$ is finite. That is, the countably infinite case of the Schur product theorem immediately reduces to the finite case, which is classical; see \url{https://en.wikipedia.org/wiki/Schur_product_theorem} for three distinct proofs. (A further proof works by observing that $S \circ T$ is a principal submatrix of the tensor product $S \otimes T$, which in turn is easily seen to be positive definite as a direct consequence of the relevant definitions.)
\qedhere

\end{proof}

Before stating Ando's theorem on Hadamard products, 
we recall that for every positive semidefinite matrix $T \in \cM^2(V)$  there exists a unique positive semidefinite matrix $T^{1/2} \in \cM^2(V)$, known as the \textbf{principal square root} of $T$, such that $(T^{1/2})^2 =T$. (Be careful to note that there may exist other matrices $S$ for which $S^2=T$ or $S S^\perp=T$, but any such matrix other than $T^{1/2}$ will not be  positive semidefinite.) Moreover, the principal square root $T^{1/2}$ of $T$ commutes with $T$ under multiplication in the sense that $T^{1/2}T=TT^{1/2}$. Indeed, the matrix $T^{1/2}T$ is also equal to the principal square root of $T^3$ and is denoted simply by $T^{3/2}$. Similarly, one may define for each $\alpha \geq 0$ a symmetric, positive semidefinite matrix $T^\alpha$ such that $(T^\alpha)_{\alpha\geq 0}$ is a semigroup in the sense that $T^\alpha$ depends continuously on $\alpha$ and $T^\alpha T^\beta=T^\beta T^\alpha = T^{\alpha+\beta}$ for every $\alpha,\beta \geq 0$. All of these claims are immediate consequences of the spectral theorem for bounded self-adjoint operators, see e.g.\ \cite[Chapter VII]{MR751959} for further background. Similar spectral considerations imply that every positive definite matrix $T \in \cM^2(V)$ has a well-defined inverse $T^{-1} \in \cM^2(V)$ with $\|T^{-1}\|_{2\to2}^{-1}=\sup\{ \lambda: \langle Tf,f\rangle \geq \lambda \langle f,f\rangle$ for every $f\in L^0(V)\}$.


\medskip

We are now ready to state the aforementioned theorem of Ando, which is a special case of \cite[Theorem 12]{MR535686}. While Ando stated his theorem for finite-dimensional matrices, the proof extends easily  to the infinite-dimensional case.
 We remark  that the paper of Ando contains a huge number of further related inequalities.


\begin{theorem}[Ando]
\label{thm:Ando}
Let $V$ be a countable set, and let $S,T \in \cM^2(V)$ be commuting, positive semidefinite matrices. Then 
\[
(ST)^{1/2} \circ (ST)^{1/2} \leq S \circ T.
\]
\end{theorem}

As discussed above, we include a full proof of this theorem with the dual aims of making it accessible to probabilists and verifying the infinite-dimensional case.
Before starting the proof, let us note that the theorem immediately implies \cref{prop:AndoTp}. 

\begin{proof}[Proof of \cref{prop:AndoTp} given \cref{thm:Ando}]
Let $T \in \cM^2(V)$ be positive semidefinite. Applying \cref{thm:Ando} to the commuting, positive semidefinite matrices $T$ and $T^3$ yields that $T^2 \circ T^2 \leq T \circ T^3$ as claimed.
\end{proof}

The proof of \cref{thm:Ando} will proceed by applying the Schur product theorem to various judiciously defined \emph{block matrices}; see \cite{MR2284176} for further applications of this technique. Let $V$ be a countable set and consider the disjoint union $V \sqcup V= V \times \{1,2\}$. Given $A,B,C,D \in \cM(V)$, we define the \textbf{block matrix}
\[
\left[\begin{matrix} A & B \\ C & D \end{matrix}\right] \in \cM(V \sqcup V)
\qquad \text{by} \qquad 
\left[\begin{matrix} A & B \\ C & D \end{matrix}\right]\bigl((u,i),(v,j)\bigr) = \begin{cases}
A(u,v) & i=1,\, j=1\\
C(u,v) & i=2,\, j=1\\
B(u,v) & i=1,\, j=2\\
D(u,v) & i=2,\, j=2
\end{cases}
\]
so that if $f,g \in L^0(V)$ and we define
\[\left[\begin{matrix} f \\ g \end{matrix}\right] \in L^0(V \sqcup V) \quad \text{by} \quad \left[\begin{matrix} f \\ g \end{matrix}\right](v,i) = \begin{cases} f(v) & i=1 \\ g(v) & i=2 \end{cases}
\qquad \text{then} \qquad
\left[\begin{matrix} A & B \\ C & D \end{matrix}\right] \left[\begin{matrix} f \\ g \end{matrix}\right] = \left[\begin{matrix} A f + Bg \\ Cf + Dg \end{matrix}\right].
\]
Similarly, if $A,B,C,D,E,F,G,H \in \cM^2(V)$ then we can compute the products of the associated block matrices in the usual way
\[
\left[\begin{matrix} A & B \\ C & D \end{matrix}\right]\left[\begin{matrix} E & F \\ G & H \end{matrix}\right] = \left[\begin{matrix} A E + BG & AF + BH \\ C E + D G & CG + DH \end{matrix}\right] \in \cM^2(V \sqcup V).
\]

\medskip

For each $T \in \cM(V)$, we define $T^\perp $ to be the transpose matrix $T^\perp(u,v)=T(v,u)$. The following lemma is an infinite-dimensional version of \cite[Theorem 1.3.3]{MR2284176}.

\begin{lemma}
\label{lem:block1}
Let $V$ be a countable set, and let $S,T,X \in \cM^2(V)$ be such that $S$ is positive semidefinite and $T$ is positive definite.
 Then 
 \[\left[\begin{matrix} S & X \\ X^\perp & T \end{matrix}\right] \geq 0 \qquad \text{ if and only if } \qquad S \geq X T^{-1} X^\perp.\]
\end{lemma}

Before proving this lemma, let us note that if $S,T \in \cM^2(V)$ are \emph{congruent} in the sense that there exists a matrix $C \in \cM^2(V)$ with inverse $C^{-1} \in \cM^2(V)$ such that $S = CTC^\perp$, then $S$ is positive semidefinite if and only if $T$ is. Indeed, first note that congruence is an equivalence relation since if $S = CTC^\perp$ then $T = C^{-1} S (C^{-1})^\perp$. Congruence is also trivially seen to preserve symmetry. To conclude, note that if $T$ is positive semidefinite then  we have that $\langle Sf,f \rangle = \langle CTC^\perp f,f \rangle = \langle T (C^\perp f), (C^\perp f)\rangle \geq 0$ for every $f\in L^2(V)$ so that $S$ is positive semidefinite as claimed. 

\begin{proof}[Proof of \cref{lem:block1}]
Since $T$ is positive definite, it is invertible with inverse $T^{-1}\in \cM^2(V)$.
Write $I$ and $O$ for the identity and zero matrices indexed by $V$ respectively. 
The block matrix we are interested in is congruent to
\begin{align}
\left[\begin{matrix} I & -XT^{-1} \\ O & I \end{matrix}\right] \left[\begin{matrix} S & X \\ X^\perp & T \end{matrix}\right]  \left[\begin{matrix} I & -XT^{-1} \\ O & I \end{matrix}\right]^\perp
&=
\left[\begin{matrix} I & -XT^{-1} \\ O & I \end{matrix}\right] \left[\begin{matrix} S & X \nonumber\\ X^\perp & T \end{matrix}\right]  \left[\begin{matrix} I & O \\ -T^{-1}X^\perp & I \end{matrix}\right]\\
&= \left[ \begin{matrix} S - X T^{-1} X^\perp & O \\ O & T \end{matrix}\right],
\label{eq:congruence}
\end{align}
where we note that this is indeed a congruence since
\[
\left[\begin{matrix} I & -XT^{-1} \\ O & I \end{matrix}\right] \qquad \text{ is invertible with inverse } \qquad \left[\begin{matrix} I & -XT^{-1} \\ O & I \end{matrix}\right]^{-1} = \left[\begin{matrix} I & XT^{-1} \\ O & I \end{matrix}\right].
\]
 The matrix on the right hand side of \eqref{eq:congruence} is positive semidefinite if and only if
\[
\left \langle \left[ \begin{matrix} S - X T^{-1} X^\perp & O \\ O & T \end{matrix}\right] \left[\begin{matrix} f \\ g \end{matrix}\right], \left[\begin{matrix} f \\ g \end{matrix}\right] \right \rangle = \bigl\langle (S -X T^{-1}X^\perp) f,f\bigr\rangle + \langle T g, g\rangle \geq 0,
\]
for every $f,g \in L^0(V)$, and since $T$ is positive definite this is the case if and only if $S -X T^{-1}X^\perp$ is positive semidefinite as claimed.
\end{proof}

We will also use the following easy fact. Note that the implication here is only in one direction.

\begin{lemma}
\label{lem:block2}
Let $V$ be countable and let $S,T \in \cM(V)$ be symmetric matrices. 
\begin{enumerate}
  \item
   If the block matrix $\left[\begin{matrix} S & T \\ T & S \end{matrix}\right]$ is positive semidefinite  then $S - T$ is positive semidefinite.
  \item If the block matrix $\left[\begin{matrix} S & T \\ T & S \end{matrix}\right]$ is positive definite  then $S - T$ is positive definite.
\end{enumerate}
\end{lemma}

\begin{proof}[Proof of \cref{lem:block2}]
We can compute that
\[\left\langle \left[\begin{matrix} S & T \\ T & S \end{matrix}\right] \left[\begin{matrix} \phantom{-}f \\ -f \end{matrix}\right] , \left[\begin{matrix} \phantom{-}f \\ -f \end{matrix}\right] \right \rangle = 2\bigl\langle (S-T) f, f \bigr\rangle
\]
for every $f \in L^0(V)$, and the claim follows from the definitions. 
\end{proof}

We are now ready to prove \cref{thm:Ando}.

\begin{proof}[Proof of \cref{thm:Ando}]
The assumption that $S$ and $T$ commute implies that $(ST)^{1/2} = (TS)^{1/2} = ((ST)^{1/2})^\perp$. In particular, we have that $S = (ST)^{1/2} T^{-1} (ST)^{1/2}$ and $T = (ST)^{1/2} S^{-1} (ST)^{1/2}$, and we deduce from \cref{lem:block1} that the block matrices \[\left[\begin{matrix} S & (ST)^{1/2} \\ (ST)^{1/2} & T \end{matrix}\right] \qquad \text{and} \qquad \left[\begin{matrix} T & (ST)^{1/2} \\ (ST)^{1/2} & S \end{matrix}\right]\] 
are both positive semidefinite. Applying the Schur product theorem we deduce that
\[
\left[\begin{matrix} S \circ T & (ST)^{1/2} \circ (ST)^{1/2} \\ (ST)^{1/2} \circ (ST)^{1/2} & S \circ T \end{matrix}\right]
\]
is positive semidefinite also, so that the claim follows from \cref{lem:block2}.
\end{proof}


\noindent \textbf{The trace on the group von Neumann algebra.}
We now wish to apply Ando's theorem to prove \cref{prop:triangle_comparison}. To do this, we will employ the notion of the \emph{trace} on the \emph{von Neumann algebra} associated to the action of $\Aut(G)$ on $G$. We assure the unfamiliar reader that, despite the intimidating name, this is in fact a very simple notion.

\medskip

Given a connected, transitive weighted graph $G=(V,E,J)$ with automorphism group $\Aut(G)$, we define the von Neumann algebra $\cA(G)$ to be the set of matrices $T \in \cM^2(V)$ that are diagonally invariant under the action of $\Aut(G)$ on $V$ in the sense that $T(\gamma u,\gamma v)=T(u,v)$ for every $u,v \in V$ and $\gamma \in \Aut(G)$. Note that $\cA(G)$ is indeed a von Neumann algebra in the sense that it is a weak$^*$-closed subspace of $\cM^2(V) \cong \cB(L^2(V))$ that is closed under multiplication and adjunction. It follows trivially from the definitions that $\cA(G)$ is also closed under Hadamard products. We will also need the slightly less obvious fact that $\cA(G)$ is closed under taking principal square roots of its positive semidefinite elements, which we now prove.


\begin{lemma}
\label{lem:square_roots_in_A}
Let $G=(V,E,J)$ be a connected, transitive weighted graph. If $T \in \cA(G)$ is positive semidefinite then its principal square root $T^{1/2}$ also belongs to $\cA(G)$.
\end{lemma}

\begin{proof}[Proof of \cref{lem:square_roots_in_A}]
It suffices to prove that $T^{1/2}$ is diagonally invariant, i.e., that $T^{1/2}(\gamma u, \gamma v) = T^{1/2}(u,v)$ for every $\gamma \in \Aut(G)$ and $u,v\in V$.
Fix $\gamma \in \Aut(G)$ and let $S \in \cM(V)$ be defined by $S(u,v)=T^{1/2}(\gamma u,\gamma v)$ for each $u,v\in V$. Since $\gamma \in \Aut(G)$ was arbitrary, it suffices to prove that $S$ is equal to $T^{1/2}$. Since the principal square root of $T$ is unique, it suffices to prove that $S$ is positive semidefinite and satisfies $S^2=T$. We begin by proving that $S$ is bounded and positive semidefinite, noting that is is clearly symmetric.
For each $f \in L^0(V)$ let $\gamma f \in L^0(V)$ be given by $\gamma f (u)=f(\gamma^{-1}u)$. Since $\gamma$ is a bijection we have that
\[
Sf(u)=\sum_{v\in V} S(u,v)f(v) = \sum_{v\in V} T^{1/2}(\gamma u, \gamma v)f( v) = \sum_{v\in V} T^{1/2}(\gamma u, v)f( \gamma^{-1} v) = [T^{1/2}(\gamma f)](\gamma u)
\]
for every $f\in L^0(V)$ and $u\in V$ and hence that
\begin{align*}
\|Sf\|=\|T^{1/2}(\gamma f)\|
\leq \|T\|_{2\to2}^2 \|\gamma f\|=\|T\|_{2\to2}^2 \|f\|
\end{align*}
and
\begin{align*}
\langle Sf,f\rangle &= \sum_{u\in V} \sum_{v\in V} T^{1/2}(\gamma u, v)f(u)f( \gamma^{-1} v) = \sum_{u\in V} \sum_{v\in V} T^{1/2}( u, v)f( \gamma^{-1} u)f( \gamma^{-1} v) 
= 
\langle T^{1/2}(\gamma f),\gamma f\rangle
\end{align*}
for every $f\in L^0(V)$.
Since $T^{1/2}$ is bounded and positive semidefinite it follows that $S$ is bounded and positive semidefinite also.
We now verify that $S^2=T$ by computing that
\begin{align*}
S^2(u,v) &= \sum_{w\in V} S(u,w)S(w,v) = \sum_{w\in V} T^{1/2}(\gamma u, \gamma w)T^{1/2}(\gamma w, \gamma v) \\ &= \sum_{w\in V} T^{1/2}(\gamma u, w)T^{1/2}(w, \gamma v) 
=T(\gamma u,\gamma v)=T(u,v)
\end{align*}
for every $u,v\in V$. We deduce that $S=T^{1/2}$ by uniqueness of the principal square root of $T$. 
\end{proof}

We now define the \emph{trace} on $\cA(G)$ and establish its basic properties.

\begin{prop}[The trace on the von Neumann algebra $\cA(G)$]
\label{prop:trace_properties}
Let $G=(V,E,J)$ be a \emph{unimodular}, connected, transitive weighted graph. Define $\Tr:\cA(G) \to \R$  by 
\[
\Tr(T) = T(o,o) \qquad \text{ for every $T\in \cA(G)$.}
\]
Then $\Tr$ is a \textbf{tracial state} on $\cA(G)$. This means that the following conditions hold:
\begin{enumerate}
\item $\Tr:\cA(G) \to \R$ is linear.
\item $\Tr(I)=1$, where $I(u,v)=\mathbbm{1}(u=v)$ is the identity matrix.
\item $\Tr(T T^\perp) \geq 0$ for every $T \in \cA(G)$, with equality if and only if $T=0$.
\item $\Tr(ST) = \Tr(TS)$ for every $S,T \in \cA(G)$.
\end{enumerate}
\end{prop}

See e.g.\ \cite[\S 5]{AL07} and \cite[Chapter 10.8]{LP:book} for previous uses of similar notions in probabilistic contexts.

\medskip

The proof of this proposition will use the \emph{signed} mass-transport principle, which states that if $G=(V,E,J)$ is a unimodular transitive weighted graph and $F:V\times V\to \R$ is diagonally invariant in the sense that $F(\gamma u,\gamma v)=F(u,v)$ for every $u,v\in V$ and $\gamma \in \Aut(G)$ 
then
$\sum_{v\in V}F(o,v) = \sum_{v\in V}F(v,o)$ provided that $F$ satisfies the integrability condition
\begin{equation}
\label{eq:MTP_integrability}
\sum_{v\in V}|F(o,v)| = \sum_{v\in V}|F(v,o)| <\infty.
\end{equation}
This can be verified by applying the usual mass-transport principle \eqref{eq:MTP} to the positive and negative parts of $F$ separately.

\begin{proof}[Proof of \cref{prop:trace_properties}]
The first two items are trivial. For the third, simply note that
\[
\Tr(TT^\perp) =  \sum_{v\in V} T(o,v)T^\perp(v,o) =  \sum_{v\in V} T(o,v)^2 
\]
for every $T \in \cA(G)$, from which the claim follows trivially.
%
Finally, for the fourth item, we apply the mass-transport principle to the diagonally invariant function $F(u,v) = S(u,v)T(v,u)$, which satisfies the integrability hypothesis \eqref{eq:MTP_integrability} since $S,T \in \cM^2(V)$, to deduce that
\begin{align*}
\Tr(ST) &=  \sum_{v\in V} S(o,v)T(v,o)
= \sum_{v\in V} S(v,o)T(o,v) = \Tr(TS). 
\end{align*}
as claimed.
\end{proof}


The following is an adaptation to our setting of an inequality due to Fej\'er \cite[Theorem A.8]{MR2064921}.

\begin{lemma}
\label{thm:Fejer}
Let $G=(V,E,J)$ be a unimodular, connected, transitive weighted graph. If $S,T \in \cA(G)$ are positive semidefinite then
$\Tr(ST) \geq 0$.
\end{lemma}

It follows in particular from this lemma that if two matrices $S,T\in \cA(G)$ satisfy $S\leq T$ then $\Tr(SX)\leq \Tr(TX)$ for every positive semidefinite matrix $X\in \cA(G)$.

\begin{proof}[Proof of \cref{thm:Fejer}]
We have by Item 4.\ of \cref{prop:trace_properties} that
\[\Tr(ST)=\Tr(S^{1/2}S^{1/2}T^{1/2}T^{1/2}) = \Tr(S^{1/2}T^{1/2} T^{1/2}S^{1/2}). \]
Since $(S^{1/2} T^{1/2})^\perp = (T^{1/2})^\perp(S^{1/2})^\perp = T^{1/2}S^{1/2}$, it follows from item 3 of \cref{prop:trace_properties} that $\Tr(ST) \geq 0$ as claimed.
%
%
\end{proof}

We are finally ready to prove that $B_\beta \leq A_\beta$ and conclude the proof of \cref{prop:triangle_comparison}.

\begin{proof}[Proof of \cref{prop:triangle_comparison}]
Let $G=(V,E,J)$ be a unimodular, connected, transitive weighted graph and let $0\leq \beta <\beta_c$.
We have by \cref{lem:Abound} that $A_\beta \leq \nabla_\beta^{2}$, so that it suffices to prove that $B_\beta \leq A_\beta$. The two-point matrix $T_\beta$ is positive definite by \cref{lem:positive_definite} and belongs to $\cM^2(V)$ and hence to $\cA(G)$ by sharpness of the phase transition as discussed above. It follows from \cref{prop:AndoTp} that $T_\beta^2 \circ T_\beta^2 \leq T_\beta \circ T_\beta^3$ and hence by \cref{thm:Fejer} that
\begin{equation*}
B_\beta =  \bigl[T_\beta^2 \circ T_\beta^2\bigr]T_\beta^2(o,o) = \Tr\Bigl(\bigl[T_\beta^2 \circ T_\beta^2\bigr]T_\beta^2\Bigr) \leq \Tr\Bigl(\bigl[T_\beta \circ T_\beta^3\bigr]T_\beta^2\Bigr)
=  \bigl[T_\beta \circ T_\beta^3\bigr]T_\beta^2(o,o) = A_\beta
\end{equation*}
as claimed. Together with \cref{prop:chi_diff_ineq} this also concludes the proof of \cref{thm:main_chi}.
\end{proof}

\medskip

\noindent \textbf{Data availability statement}: Data sharing not applicable to this article as no datasets were generated or analysed during the current
study.

\medskip

\noindent \textbf{Conflicts of interest statement}: The authors have no financial or proprietary interests in any material discussed in this article.

\end{document}

%% file: header.tex
\usepackage{amsmath,amssymb,amsfonts,amsthm}
\usepackage{color}

\usepackage[colorlinks=true,linkcolor=blue,citecolor=blue,urlcolor=blue,pdfborder={0 0 0}]{hyperref}
\usepackage{cleveref}

\usepackage{caption}
\usepackage{thmtools}

\usepackage{mathrsfs}

\crefname{theorem}{Theorem}{Theorems}
\crefname{thm}{Theorem}{Theorems}
\crefname{lemma}{Lemma}{Lemmas}
\crefname{lem}{Lemma}{Lemmas}
\crefname{remark}{Remark}{Remarks}
\crefname{prop}{Proposition}{Propositions}
\crefname{defn}{Definition}{Definitions}
\crefname{corollary}{Corollary}{Corollaries}
\crefname{conjecture}{Conjecture}{Conjectures}
\crefname{question}{Question}{Questions}
\crefname{chapter}{Chapter}{Chapters}
\crefname{section}{Section}{Sections}
\crefname{figure}{Figure}{Figures}

\theoremstyle{plain}
\newtheorem{thm}{Theorem}[section]

\newtheorem{lemma}[thm]{Lemma}
\newtheorem{theorem}[thm]{Theorem}

\newtheorem{corollary}[thm]{Corollary}

\newtheorem{prop}[thm]{Proposition}

\theoremstyle{definition}

\theoremstyle{remark}
\newtheorem{remark}[thm]{Remark}

\numberwithin{equation}{section}

%% file: commands.tex
\renewcommand{\P}{\mathbb P}
\newcommand{\E}{\mathbb E}
\newcommand{\C}{\mathbb C}
\newcommand{\R}{\mathbb R}
\newcommand{\Z}{\mathbb Z}
\newcommand{\N}{\mathbb N}

\newcommand{\cA}{\mathcal A}
\newcommand{\cB}{\mathcal B}

\newcommand{\cM}{\mathcal M}

\newcommand{\sA}{\mathscr A}

\newcommand{\sC}{\mathscr C}

\newcommand{\bbH}{\mathbb H}

